\documentclass[12pt]{article}
\usepackage{amsmath}
\usepackage{amssymb,amscd}
\usepackage{bm}
\usepackage{wrapfig}

\usepackage{theorem}

\newtheorem{theorem}{Theorem}[section]
 \newtheorem{definition}[theorem]{Definition}
 \newtheorem{proposition}[theorem]{Proposition}

{\theorembodyfont{\normalfont}
 \newtheorem{example}{Example}[section]
 
 }

\def\thefootnote{\fnsymbol{footnote}}

\newlength{\minitwocolumn}
\catcode`\@=11
\long\def\@makefntext#1{
\protect\noindent \hbox to 3.2pt {\hskip-.9pt
$^{{\eightrm\@thefnmark}}$\hfil}#1\hfill}               

\def\thefootnote{\fnsymbol{footnote}}
\def\@makefnmark{\hbox to 0pt{$^{\@thefnmark}$\hss}}    

\def\ps@myheadings{\let\@mkboth\@gobbletwo
\def\@oddhead{\hbox{}
\rightmark\hfil\eightrm\thepage}
\def\@oddfoot{}\def\@evenhead{\eightrm\thepage\hfil
\leftmark\hbox{}}\def\@evenfoot{}
\def\sectionmark##1{}\def\subsectionmark##1{}}

\font\eightrm=cmr8

\font\sc=cmr5 scaled\magstep1

\newenvironment{proof}[1][Proof]{\begin{trivlist}
\item[\hskip \labelsep {\bfseries #1}]}{\end{trivlist}}
\newcommand{\qed}{\nobreak \ifvmode \relax \else
      \ifdim\lastskip<1.5em \hskip-\lastskip
      \hskip1.5em plus0em minus0.5em \fi \nobreak
      \vrule height0.75em width0.5em depth0.25em\fi}

\newcommand\hQ{\mbox{\boldmath $Q$}}

\newcommand{\sbv}[2]{{\{{{#1},{#2}}\}}}
\newcommand{\tbv}[2]{{\{{{#1},{#2}}\}_{\Theta}}}

\newcommand{\floor}[1]{{\lfloor #1 \rfloor}}

\def\bz{\mbox{\boldmath $Z$}}
\def\bx{\mbox{$x$}}

\def\bq{\mbox{\boldmath $q$}}
\def\bp{\mbox{\boldmath $p$}}

\def\bbx{\mbox{\boldmath $x$}}
\def\bbxi{\mbox{\boldmath $\xi$}}
\def\bbq{\mbox{\boldmath $q$}}
\def\bbp{\mbox{\boldmath $p$}}

\def\bbd{\mbox{\boldmath $d$}}

\def\bu{\mbox{\boldmath $u$}}
\def\bv{\mbox{\boldmath $v$}}
\def\bw{\mbox{\boldmath $w$}}
\def\bz{\mbox{\boldmath $z$}}

\def\bomega{\mbox{\boldmath $\omega$}}

\def\bZ{\mbox{\boldmath $Z$}}

\newcommand{\calD}{{\cal D}}

\newcommand{\calL}{{\cal L}}
\newcommand{\calM}{{\cal M}}

\newcommand{\calX}{{\cal X}}

\newcommand{\Map}{{\rm Map}}
\newcommand{\ev}{{\rm ev}}
\newcommand{\AKSZ}{{\bf  AKSZ}}
\newcommand{\tAKSZ}{{\bf  tAKSZ}}
\newcommand{\QPM}{{\bf QP}}
\newcommand{\tQPM}{{\bf tQP}}

\newcommand{\beq}{\begin{equation}}
\newcommand{\eeq}{\end{equation}}
\newcommand{\bea}{\begin{eqnarray*}}
\newcommand{\eea}{\end{eqnarray*}}
\newcommand{\beqa}{\begin{eqnarray}}
\newcommand{\eeqa}{\end{eqnarray}}

\begin{document}


\baselineskip 0.7cm

\begin{titlepage}
\begin{flushright}
\end{flushright}

\vskip 1.35cm
\begin{center}
{\Large \bf
Canonical Functions, 
Differential Graded Symplectic Pairs
in Supergeometry, and 
AKSZ Sigma Models
with Boundaries
}
\vskip 1.2cm
Noriaki IKEDA$^1$%
\footnote{E-mail:\
nikeda@se.ritsumei.ac.jp,
ikeda@yukawa.kyoto-u.ac.jp}
and
Xiaomeng XU$^2$
\footnote{E-mail:\
Xiaomeng.Xu@unige.ch
}
\vskip 0.4cm
{\it
$^1$
Department of Mathematical Sciences,
Ritsumeikan University \\
Kusatsu, Shiga 525-8577, Japan \\
and \\
Maskawa Institute for Science and Culture,
Kyoto Sangyo University, \\
Kyoto 603-8555, Japan \\
$^2$
Section of Mathematics, University of Geneva \\
2-4 Rue de Li$\rm\grave{e}$vre, 
c.~p.~64, 1211-Gen$\rm\grave{e}$ve
4, Switzerland}
\vskip 0.4cm

\today

\vskip 1.5cm

\begin{abstract}
Consistent boundary conditions for 
Alexandrov-Kontsevich-Schwartz-Zaboronsky (AKSZ) sigma models and the corresponding boundary theories are 
analyzed.
As their mathematical structures,
we introduce 
a generalization of
differential graded symplectic manifolds, called twisted QP manifolds,
in terms of graded symplectic geometry,
canonical functions,
and QP pairs.
We generalize the AKSZ construction of topological sigma models
to sigma models with Wess-Zumino terms 
and show that all the twisted Poisson-like structures known in the literature 
can actually be naturally realized as boundary conditions 
for AKSZ sigma models.
\end{abstract}
\end{center}
\end{titlepage}

\renewcommand{\thefootnote}{\alph{footnote}}

\setcounter{page}{2}


\rm

\section{Introduction}
\noindent
A differential graded symplectic manifold (QP manifold) 
has been introduced from analysis of the Batalin-Vilkovisky formalism 
\cite{Schwarz:1992nx}.  
It was later used for a simple procedure for constructing
topological sigma models
by Alexandrov, Kontsevich, Schwarz, and Zaboronsky 
\cite{Alexandrov:1995kv}. 
Following this, graded symplectic geometry
has been of interest in both mathematics and physics 
due to its rich mathematical structures and links with a variety of topics; 
see 
\cite{Cattaneo:2010re}\cite{Roy01}\cite{Vaintrob}\cite{Voronov:2001qf}.


A Poisson manifold is naturally a QP manifold of degree $1$. 
A Courant algebroid, which was introduced in \cite{lwx} 
to describe the double of Lie bialgebroids, 
has a one-to-one correspondence to 
a QP manifold of degree $2$ \cite{Roy01}. 
The corresponding Alexandrov-Kontsevich-Schwartz-Zaboronsky (AKSZ) sigma models are the Poisson sigma model 
\cite{Ikeda:1993fh}\cite{Ikeda:1993aj}\cite{Schaller:1994es}
and the Courant sigma model \cite{Ikeda:2002wh}\cite{Roytenberg:2006qz}, 
respectively. 
Weaker versions of Poisson structures and variants of 
Lie algebroids and Courant algebroids are motivated by questions 
from quantum groups and field theories. 
For example, an extension of the Poisson sigma model
by a Wess-Zumino (WZ) term
naturally leads to WZ-Poisson manifolds
\cite{Klimcik:2001vg}\cite{Park2000au}, or equivalently,
the Poisson structures with a $3$-form background \cite{Severa:2001qm}. 
A similar structure was studied in the framework of 
the Courant algebroid theory.
Hansen and Strobl \cite{Hansen:2009zd} showed that 
a generalization of Courant algebroids with a $4$-form
arises naturally in the Courant sigma model with a Wess-Zumino term.

Our main result shows that all the twisted Poisson-like structures 
known in the literature can actually be naturally realized as 
(and so in a certain sense ``are'') boundary conditions for 
AKSZ sigma models. 
This result greatly clarifies the meaning of the otherwise 
rather obscure conditions defining a twisted Poisson-like structure.

To do this, we introduce a canonical function
and a QP pair (a differential graded symplectic pair), 
which are generalizations of
differential graded structures in graded symplectic geometry.
This idea is inspired by analysis of the 
consistency of boundary conditions of AKSZ sigma models 
\cite{Alexandrov:1995kv}\cite{Cattaneo:2001ys}\cite{Ikeda:2012pv}\cite{Roytenberg:2006qz}, which are
topological sigma models constructed by supergeometric methods
and a topological open
membrane \cite{Hofman:2002rv}\cite{Park2000au}. 
A similar structure appears in the Batalin-Vilkovisky (BV) 
formalism of string field theory
\cite{Hata:1993gf}.

Another motivation comes from the canonical transformation in symplectic 
supergeometry, which is also called twisting
\cite{Roytenberg:2001am}. 
It can be viewed as a higher analogue of
a Poisson function \cite{Kosmann-Schwarzbach:2007}\cite{Terashima}, 
and it defines a generalization of the Dirac structure
\cite{Courant}.
Moreover, a canonical function describes the boundary condition 
structures of the AKSZ sigma models,
which have played key roles in the derivation of the 
deformation quantization from the Poisson sigma model
\cite{Kontsevich:1997vb}\cite{Cattaneo:1999fm}. 

A canonical function leads to the concept of a QP pair, which is a certain 
tower of two (twisted) differential graded symplectic manifolds.
This 
unifies the various concepts that were separately analyzed
above, and it
includes many geometric structures 
such as the Lie ($2$-)algebras, the (twisted
or quasi) Poisson structures, the (homotopy) Lie algebroids, the
(twisted) Courant algebroids, the Nambu-Poisson structures and others.
Some of these will be used below as examples.

Analysis of a canonical transformation on a QP manifold
naturally leads to what we call a twisted QP manifold. 
It is a mathematical framework for a unified understanding 
of the so-called Wess-Zumino terms, together with twisted Poisson-like structures. 
A general method to get a twisted QP manifold is given 
by the deformation theory. 
Some examples will be presented to
illuminate this twisting process.
Moreover
a new geometric structure,
the strong Courant algebroid, is proposed.

The defining structure of a QP pair guarantees consistency 
of the bulk structure
and the boundary conditions.
In general, a quantum theory on a manifold $X$ in
$n+1$ dimensions with given boundaries 
may have the same structure as the corresponding quantum
theory on the boundary $\partial X$ in $n$ dimensions 
\cite{Bousso:2002db}.
When applying this so-called bulk-boundary correspondence
of quantum field theories to the AKSZ sigma models, 
we find that it is necessary to extend the AKSZ sigma models
to the `twisted' AKSZ sigma models.
This is a generalization of the Chern-Simons/WZW correspondence.
We propose an extension of the AKSZ construction to a twisted version
via a map from twisted QP manifolds to topological sigma models.

There are recent works on similar topics. 
Twisted structures in the setting of $L_{\infty}$-algebras are analyzed in \cite{FregierZambon1}\cite{FregierZambon2}.
There are studies that present formulations of 
AKSZ sigma models with boundaries
\cite{Cattaneo:2012qu}\cite{Cattaneo:2012zs}
\cite{Fiorenza:2011jr}\cite{Fiorenza:2013nha}
\cite{Ikeda:2013vga}. 
The concept of QP pairs is used in the current algebra theory in 
\cite{Ikeda:2013vga}. 
In \cite{xiaomeng}, the twisted QP manifold introduced in this paper is shown to be a
homotopy version of a QP manifold in the spirit of a homotopy Poisson
manifold. 
The AKSZ sigma models 
have been reformulated by using derived geometry 
\cite{Calaque}\cite{PTVV}.

This paper is organized as follows.
In Section 2, we review
a QP manifold and an AKSZ sigma model without boundaries.
In Section 3, 
we analyze the AKSZ sigma models with boundaries and introduce
canonical functions.
In Section 4, we analyze
a twisted QP manifold and present
some examples. In particular, we propose a new 
algebroid, the (twisted) strong Courant algebroid.
In Section 5, 
we consider a physical application of a twisted QP manifold 
whose defining structure guarantees consistency of 
the corresponding AKSZ sigma models with boundaries.
Bulk-boundary correspondence of sigma models
leads us to the concept of twisted AKSZ sigma models.
Section 6 summarizes our results and considers areas of future work.

\section{QP Manifolds and AKSZ Sigma  Models}
\subsection{QP Manifolds}
\noindent
A graded manifold $\calM$
is a ringed space
with a structure sheaf
of 
\bZ-graded commutative
algebra over an ordinary smooth manifold $M$.
Grading is compatible with supermanifold grading,
that is, a variable of even degree is commutative, and
one of odd degree is anticommutative. 
By definition, the structure sheaf of $\calM$ is locally isomorphic to
$C^{\infty}(U) \otimes S^{\bullet}(V)$,
where
$U$ is a local chart on $M$,
$V$ is a graded vector space, and $S^{\bullet}(V)$ is a free
graded commutative ring on $V$.
Refer to \cite{Carmeli}\cite{Varadarajan}
for the rigorous mathematical definition of objects
in supergeometry.

An N-manifold (i.e.,~a nonnegatively graded manifold) $\calM$ 
equipped with
a graded symplectic structure
$\omega$ of degree $n$ is called a {\it P-manifold} of degree $n$,
denoted by $({\calM},\omega)$.
$\omega$ is also called a $P$-structure.
A graded Poisson bracket on $C^\infty ({\calM})$ is defined as
$    \{f,g\} = (-1)^{|f|+n+1} i_{X_f} i_{X_g}\omega$,
where 
a Hamiltonian vector field $X_f$ is defined by the equation
$i_{X_f}\omega= - \delta f$
for any $f\in C^\infty({\calM})$, and
$\delta$ is a differential on $\calM$. 
A vector field $Q$ on $\calM$ is called homological if $Q^2=0$.
\begin{definition}
A \textbf{QP-manifold} is a $P$-manifold $(\calM,\omega)$ endowed with a degree 1 homological vector field $Q$ such that $\calL_Q \omega =0$ 
\cite{Schwarz:1992nx}.
\end{definition}
We call the homological vector field $Q$ the Q-structure.
We also denote a QP manifold by
the corresponding triple $(\calM,\omega,Q)$.
For any QP manifold of positive degree, there exists 
a Hamiltonian function $\Theta\in C^{\infty}(\calM)$ of $Q$ 
with respect to the graded Poisson bracket $\{-,-\}$, that is,
\beq
Q=\{\Theta,-\}.
\eeq
Then, the homological condition, $Q^2=0$, implies that
$\Theta$ is a solution of the \textit{classical master equation},
\begin{equation}
\{\Theta,\Theta\}=0.
\label{cmaseq}
\end{equation}
Such a $\Theta$ is also called a homological function.

\subsection{AKSZ Sigma Models without Boundaries}\label{sectionAKSZwboundary}
\noindent
In this subsection, we review the AKSZ construction
\cite{Alexandrov:1995kv}\cite{Cattaneo:2001ys}\cite{Roytenberg:2006qz},
which is a systematic method for constructing
a topological sigma model from a QP manifold. 
The resulting sigma model is called an AKSZ sigma model. 

Let $(\calX, D, \mu)$ be a differential graded manifold
$\calX$ 
with a $D$-invariant nondegenerate measure $\mu$,
where
$D$ is a differential on $\calX$.
Let ($\calM, \omega, Q$) be a QP-manifold,
and let
$\Map(\calX, \calM)$ be the space of smooth maps from $\calX$ to $\calM$.

Since
${\rm Diff}(\calX)\times {\rm Diff}(\calM)$
naturally acts on $\Map(\calX, \calM)$,
$D$ and $Q$ induce differentials $\hat{D}$ and $\hat{Q}$, respectively,
on $\Map(\calX, \calM)$.
Explicitly, 
$\hat{D}(z, f) = D(z) \delta f(z)$
and $\check{Q}(z, f) = Q f(z)$,
for all $z \in \calX$ and $f \in \Map(\calX, \calM)$.

An {\it evaluation map}
${\rm ev}: \calX \times \Map(\calX, \calM) \longrightarrow \calM$
is defined as
${\rm ev}:(z, f) \longmapsto f(z)$,
for any
$z \in \calX$ and $f \in \Map(\calX, \calM)$.
A {\it chain map}
$\mu_*: \Omega^{\bullet}(\calX \times \Map(\calX, \calM))
\longrightarrow \Omega^{\bullet}(\Map(\calX, \calM))$ is defined as
$$\mu_* \omega(y)(v_1, \ldots, v_k)
 = \int_{\calX} \mu(x)
 \omega(x, y) (v_1, \ldots, v_k)$$
where $v$ is a vector field on $\calX$
and
$\int_{\calX} \mu$ is the Berezin integration on $\calX$.
The composition $\mu_* \ev^*: \Omega^{\bullet}(\calM)
\longrightarrow \Omega^{\bullet}(\Map(\calX, \calM))$
is called the {\it transgression map}.

A \textbf{P-structure} $\bomega$ (a graded symplectic structure) 
on $\Map(\calX, \calM)$ is defined by
\begin{eqnarray*}
\bomega := \mu_* \ev^* \omega.
\end{eqnarray*}
%
Note that $\bomega$ is nondegenerate and closed
because the operation $\mu_* \ev^*$ preserves these properties.
The corresponding graded Poisson bracket on $\Map(\calX, \calM)$ is 
denoted by $\sbv{-}{-}$.

A \textbf{Q-structure} function 
$S$ on $\Map(\calX, \calM)$
is constructed as follows.
$S$ consists of two parts $S:=S_0+S_1$.
We take a canonical $1$-form $\vartheta$
on $\calM$ such that
$\omega= - \delta \vartheta$ and
define $S_0 := \iota_{\hat{D}} \mu_* {\rm ev}^* \vartheta$.
Moreover, we define $S_1 := \mu_* \ev^* \Theta$, 
where $\Theta$ is the homological function on $\calM$.
Then we can prove that $S$ is a homological function 
on $\Map(\calX,  \calM)$:
\begin{eqnarray}
\sbv{S}{S} =0,
\label{classicalmaster}
\end{eqnarray}
from the definitions of $S_0$ and $S_1$
and the properties of maps.
A degree $1$ homological vector field $\hQ$ is defined as $\hQ = \sbv{S}{-}$.
The classical master equation shows that
$\hQ$ is a coboundary operator, $\hQ^2=0$.
%

We thus have the following theorem.
\begin{theorem}\cite{Alexandrov:1995kv}
Let $\calX$ be a differential graded manifold with a compatible measure, 
and let $\calM$ be a QP-manifold,
then the mapping space $\Map(\calX, \calM)$
has a QP structure.
\end{theorem}
This structure is called an AKSZ sigma model.
If $\calX = T[1]X$, where $X$ is a manifold in $n+1$ dimensions,
the QP structure on $\Map(\calX, \calM)$
is of degree $-1$.
In this case, a QP structure on $\Map(T[1]X,  \calM)$
is equivalent to the Batalin-Vilkovisky formalism
of a topological sigma model.

\section{AKSZ Sigma Models with Boundaries}
\noindent
In this section, the geometric structure of
the AKSZ sigma model on a base manifold $X$ with a boundary is
analyzed and proven to be described by a quintuple mathematical datum, 
$(\calM, \omega, \Theta, \calL, \alpha)$.
When $n=1$, this corresponds to a topological open string
and produces a deformation quantization formula
\cite{Cattaneo:1999fm}.
When $n \geq 2$, such theories describe topological open membranes
\cite{Park2000au}\cite{Hofman:2002rv}.
\subsection{Consistent Boundary Conditions}
\noindent
Let us take a manifold $X$ in $n+1$ dimensions with nonempty boundaries.
Let $(\calX = T[1]X, D, \mu)$ be a differential graded manifold over $X$
with a differential $D$ and a compatible measure $\mu$,
and let
$(\calM, \omega, \Theta)$ be a QP manifold of degree $n$.
Then the AKSZ construction produces a consistent topological sigma model on the mapping space 
$\Map(T[1]X, \calM)$, as long as the boundary conditions on $\partial \calX$ are consistent with the QP structure of the whole theory.
Generally the classical master equation is not satisfied because of the boundary terms. By Stokes' theorem, a straightforward computation gives
\begin{eqnarray}
\sbv{S}{S}= \int_{\partial \calX}
\mu_{\partial \calX} \
 (i_{\partial} \times {\rm id})^* \ 
\ev^* (\vartheta + \Theta),
\label{boundaryCME}
\end{eqnarray}
where
$\mu_{\partial \calX} $ is a boundary measure
induced from $\mu$ to $\partial \calX$ by
the inclusion map
$i_{\partial}: \partial \calX \longrightarrow \calX$.
Since the classical master equation 
$\sbv{S}{S}= 0$ must be satisfied for consistency of the theory,
the right-hand side of Equation (\ref{boundaryCME})
must vanish.
From this observation, we have
\begin{theorem}\label{boundarytheta}
If $(\vartheta + \Theta) |_{{\rm Im} \, \partial \calX} =0$
on $\calM$,
then $\sbv{S}{S}=0$.
\end{theorem}
Let us consider the physical constraints.
$S$ is a classical BV action in a physical theory.
In order to derive the equation of motion
from the variational principle in mechanics,
the variation of $S$ must vanish on the boundaries.
This is satisfied
if $\vartheta = 0$ on ${\rm Im} \ \partial \calX$.
Since $\omega = - \delta \vartheta$,
this says that ${\rm Im} \ \partial \calX$ belongs to
a subspace of a Lagrangian submanifold $\calL \subset \calM$
that is the zero locus of $\vartheta=0$.
This physical requirement turns Theorem \ref{boundarytheta} 
into the following condition on $\calM$:
\begin{proposition}\label{boundarytheta2}
Let $\calL$ be a Lagrangian submanifold of $\calM$, i.e.,
$\vartheta|_{\calL} =0$.
Then $\sbv{S}{S}=0$ is satisfied if
$\Theta |_{\calL} =0$.
\end{proposition}

We demonstrate that this is consistent with the variational principle by 
taking the Darboux coordinates of the superfields
with respect to the P-structure $\bomega$
on $\Map(\calX, \calM)$, where the superfields are pullbacks 
of the local coordinates on $\calM$ by
$\bbx^*$, the degree zero map $\bbx:\calX \longrightarrow M$.
We denote $\bbq^{a(i)}(\sigma, \theta)
\in \Gamma (T[1]X \otimes \bbx^*(\calM_i))$
for $0 \leq i \leq \floor{n/2}$ and
$\bbp_{a(n-i)}(\sigma, \theta)
\in \Gamma (T[1]X \otimes \bbx^*(\calM_{n-i}))$
for $\floor{n/2} < i \leq n$,
where
$(\sigma, \theta)$ are local coordinates on $\calX = T[1]X$,
$\calM_{i}$ is the degree $i$ subspace of $\calM$,
$\bbx = \bbq^{a(0)}$,
and $\floor{m}$ is the floor function,
which gives the largest integer less than
or equal to $m$.
The Poisson brackets of superfields are
\begin{eqnarray}\label{BVbracket3}
\sbv{\bbq^{a(i)}(\sigma, \theta)}
{\bbp_{b(j)}(\sigma^{\prime}, \theta^{\prime})}
&=& \delta^{i}{}_{j} \delta^{a(i)}{}_{b(j)}
\delta^{n+1}(\sigma-\sigma^{\prime})
\delta^{n+1}(\theta-\theta^{\prime}),
\end{eqnarray}
and
if $n$ is even and $i = j = n/2$,
\begin{eqnarray}\label{BVbracket4}
\sbv{\bbq^{a(n/2)}(\sigma, \theta)}
{\bbq^{b(n/2)}(\sigma^{\prime}, \theta^{\prime})}
&=& k^{a(n/2)b(n/2)}
\delta^{n+1}(\sigma-\sigma^{\prime})
\delta^{n+1}(\theta-\theta^{\prime}),
\end{eqnarray}
where $k^{a(n/2)b(n/2)}$ is a metric on $\calM_{n/2}$.
The exterior derivative $d$ on $X$ induces
a superdifferential
$\bbd = \theta^{\mu} \frac{\partial}{\partial \sigma^{\mu}}$
on $\Map(\calX, \calM)$.
If we define
$\bbp_{a(n/2)} = k_{a(n/2)b(n/2)} \bbq^{a(n/2)}$,
then $S$ takes the same form for odd or even $n$,
\begin{eqnarray}
S &=& S_0 + S_1
= \int_{\calX}
\mu
\
\left(
\sum_{0 \leq i \leq \floor{n/2}}
(-1)^{n+1-i} \bbp_{a(i)} \bbd \bbq^{a(i)}
\right)
+ \int_{\calX}
\mu
\ \ev^* \Theta(\bbq, \bbp).
\label{Szero}
\end{eqnarray}

In order to derive the equations of motion
of the superfields,
the variation of $S=S_0+S_1$ must vanish on the boundaries.
The variation is 
$\delta S = 
\int_{\partial \calX}
\mu_{\partial \calX}
\
\sum_{0 \leq i \leq (n-1)/2}
(-1)^{n+1-i} \bbp_{a(i)} \delta \bbq^{a(i)}
$ 
up to the equations of motion.
Thus, we should impose the condition
$\bbp_{a(i)} \delta \bbq^{a(i)} =0$
on the boundary $\partial \calX$. 

\subsection{Canonical Transformations and Canonical Functions}
\noindent
A physical theory can have 
boundary terms if the base manifold $X$ has a boundary. 
We introduce boundary terms by a canonical transformation.

Let $(\calM, \omega, Q)$ be a QP manifold of degree $n$,
and suppose that $Q$ is generated by a
Hamiltonian function $\Theta \in C^{\infty}(\calM)$
of degree $n+1$.
Let $\alpha \in C^{\infty}(\calM)$. We define 
an exponential operation $e^{\delta_\alpha}$ by
$$
e^{\delta_\alpha} f
= f + \sbv{f}{\alpha}
+ \frac{1}{2} \sbv{\sbv{f}{\alpha}}{\alpha}
+ \cdots,
$$
for any $f\in C^{\infty}(\cal M)$.

\begin{definition}
For any function $\alpha$ of degree $n$, 
$e^{\delta_\alpha}$ is called a \textbf{canonical transformation}. 
\end{definition}
Note that $e^{\delta_\alpha}$ was called \textbf{twisting} 
in \cite{Roytenberg:2001am}.
In addition, for any $\alpha$ of degree $n$, 
$\sbv{e^{\delta_\alpha} f}{e^{\delta_\alpha} g} 
= e^{\delta_\alpha} \sbv{f}{g}$, where $f, g \in C^{\infty}(\calM)$.
Now we introduce the canonical function.

\begin{definition}\label{CanonicalFunction}
Let $(\calM, \omega, Q=\sbv{\Theta}{-})$
be a QP manifold of degree $n$.
A function $\alpha$ of degree $n$
is called a \textbf{canonical function} with respect to $\calL$ if
$e^{\delta_\alpha} \Theta |_{\calL} =0$, where $\calL$ is 
a Lagrangian submanifold
of $(\calM,\omega)$ and $|_{\calL}$ is the restriction on $\calL$.
\end{definition}
Obviously if $\alpha$ is a canonical function on a QP manifold
$(\calM,\omega,\Theta)$, $e^{\delta_\alpha}$ is a canonical
transformation.

We make a canonical transformation $e^{\delta_\alpha}$
by a function $\alpha$ of degree $n$.
This changes a target QP manifold to
$(\calM, \omega, \Theta_{\alpha})$, where
the Q-structure is $\Theta_{\alpha} = e^{\delta_\alpha} \Theta$.
Since the P-structure does not change, 
a new Q-structure $S^{\prime}$ on $\Map(T[1]X,\cal M)$
in the AKSZ sigma model becomes
\begin{eqnarray}
S^{\prime} &=& S_0 + S_1^{\prime}
\nonumber \\
&=& \iota_{\hat{D}} \mu_* {\rm ev}^* \vartheta
+ \mu_* \ev^*
e^{\delta_\alpha} \Theta.
\label{akszwithboundary}
\end{eqnarray}
If $\partial X = \emptyset$, 
$S^{\prime}$ also satisfies
the classical master equation
because of
$\sbv{e^{\delta_\alpha} \Theta}{e^{\delta_\alpha} \Theta}
= e^{\delta_\alpha} \sbv{\Theta}{\Theta} =0$.
If $\partial X \neq \emptyset$,
we should modify the boundary conditions
so that $S^\prime$ satisfies the classical master equation.
Applying Proposition \ref{boundarytheta2} to
Equation (\ref{akszwithboundary}), we have
\begin{proposition}\label{boundaryQstr}\cite{Hofman:2002rv}
We assume $\partial \calX \neq \emptyset$.
Let $(\calM, \omega, \Theta)$ be a QP manifold of degree $n$,
and let $\alpha \in C^{\infty}(\calM)$ of degree $n$
be a canonical transformation.
The classical master equation $\sbv{S}{S}=0$
is satisfied in an AKSZ sigma model (\ref{akszwithboundary}) if
$
e^{\delta_\alpha} \Theta |_{\calL} =0$,
that is, $\alpha$ is a canonical function 
with respect to a Lagrangian submanifold $\calL$ of $\calM$.
\end{proposition}

From Proposition \ref{boundaryQstr},
the mathematical structure of an
AKSZ sigma model with a boundary is 
a quintuple $(\calM, \omega, \Theta, \calL, \alpha)$.

The first three elements of the 
datum form a QP manifold $(\calM, \omega, \Theta)$.
The full quintuple is a QP manifold with 
a trivialization of $\omega$ and $\Theta$.
The trivialization of $\omega$, $\omega|_{\calL}=0$ 
determines a Lagrangian submanifold $\calL$ of $\calM$
(we take the maximal $\calL$).
Finally, we require a trivialization of $\Theta$, but 
this can be refined so that $\Theta$ vanishes on $\calL$ 
up to a canonical transformation $e^{\delta_{\alpha}}$, i.e., 
$e^{\delta_\alpha} \Theta |_{\calL} =0$
such that $|\alpha| = |\omega|$.

\subsection{From Canonical Functions to Boundary Terms}
\noindent
We see that $\alpha$ corresponds to physical boundary terms
in a special case.

In Equation (\ref{akszwithboundary}),
the geometric structure is equivalent 
even if we make the inverse canonical transformation 
on $\Map(\calX, \calM)$ for the P-structure in addition to that for the Q-structure.
We then have an equivalent expression for the homological function:
\begin{eqnarray}
S^{\prime\prime} &=& e^{- \delta_{\alpha}} S^{\prime}
\nonumber \\
&=& \iota_{\hat{D}} \mu_* {\rm ev}^*
e^{- \delta_{\alpha}} \vartheta
+ \mu_* \ev^*
e^{- \delta_{\alpha}} e^{\delta_\alpha} \Theta.
\nonumber \\
&=& \iota_{\hat{D}} \mu_* {\rm ev}^*
e^{- \delta_{\alpha}} \vartheta
+ \mu_* \ev^*
\Theta.
\label{ChangeS}
\end{eqnarray}
The AKSZ sigma model is $S^{\prime\prime}$, 
and its geometric structure
defines the equivalent QP structure as
the original $S^{\prime}$ on $\Map(\calX, \calM)$.

Let us consider the special case such that $\alpha$ satisfies
$\sbv{\alpha}{\alpha}=0$. 
Then, since $e^{- \delta_{\alpha}} \vartheta
= \vartheta - \sbv{\vartheta}{\alpha}$,
${\alpha}$ generates a boundary term
$S_{\partial \calX} =
\mu_{{\partial \calX}*} (i_{\partial} \times {\rm id})^* \ev^* {\alpha}$ 
by the following computation:
\begin{eqnarray}
S^{\prime\prime}
&=&
S_0
- \mu_* \, \ev^*
\sbv{\vartheta}{\alpha}
+ \mu_* \, \ev^* \Theta
\nonumber \\
&=&
S_0 
-
L_{\hat{D}} \, \mu_* \ev^* {\alpha}
+ 
\mu_* \, \ev^* \Theta
\nonumber \\
&=&
\iota_{\hat{D}} \mu_* {\rm ev}^* \vartheta
+ \mu_* \ev^* \Theta
-
\mu_{\partial \calX *} \, (i_{\partial} \times {\rm id})^*
\ev^* {\alpha}.
\label{AKSZbulk}
\end{eqnarray}
Here we have derived this equation from the formula 
$\sbv{S_0}{-}= L_{\hat{D}}(-)$ and 
the Stokes' theorem,
$L_{\hat{D}} \mu_* \ev^* (-) =
\mu_{\partial \calX *} \, 
(i_{\partial}^* \times {\rm id})^* \ \ev^* (-)$,
where $L_{\hat{D}}$ is a Lie derivative.
%
Thus a canonical function leads to
a boundary source term, 
$S_{\partial \calX} =
\mu_{\partial \calX *} \, (i_{\partial} \times {\rm id})^*
\ev^* {\alpha}$.
Physically, this sigma model describes a
topological open membrane
with boundary charges
\cite{Hofman:2002rv}\cite{Park2000au}.
Although general canonical functions can generate
more complicated boundary terms which are not
integrations of local Lagrangians,
Proposition \ref{boundaryQstr} says that
these also provide physically consistent
boundary deformations of the AKSZ sigma models.


In the next section, we will focus on 
the underlying mathematical structure,
a quintuple 
$(\calM, \omega, \Theta, \calL, \alpha)$, and
a generalization of QP manifolds,
twisted QP manifolds, will be introduced and analyzed. 
Moreover, a new geometric structure, a QP pair, will be proposed.

\section{Twisted Higher-Degree Poisson Manifolds}
\subsection{Twisted QP manifolds}
\label{twistedQPmfd}
\noindent
Let $\calM$ be a QP manifold of degree $n+1$,
and let $\calL$ be a Lagrangian submanifold.
In this section, we concentrate on the structure on $\calL$. 
By the Lagrangian neighborhood theorem in symplectic geometry, 
the boundary condition in Proposition \ref{boundaryQstr},
$(\calM, \omega, \Theta, \calL, \alpha)$, is locally  realized
by $(T^*[n+1]\calL, \omega_{can}, \Theta, \calL, \alpha)$.
Here $\omega_{can}$ denotes the canonical symplectic form 
on the shifted cotangent bundle $T^*[n+1]\calL$.
Therefore, in the following, we will discuss a QP manifold of degree $n+1$,
$(T^*[n+1]\calL, \omega, \Theta)$,
where $\calL$ is an N-manifold of degree $n$. 
We may assume $\omega = \omega_{can}$, but this is not necessary.

Let $\pi: T^*[n+1]\calL \rightarrow \calL$ be the natural projection.
The derived bracket \cite{Kosmann-Schwarzbach:2003en} $\tbv{-}{-}$ 
on $C^\infty(\calL)$:
\begin{eqnarray}
\tbv{f}{g} = \sbv{\sbv{\pi^* f}{\Theta}}{\pi^* g}|_{\calL},
\label{derivedQPbracket}
\end{eqnarray}
is a graded Poisson bracket for any $f$, $g\in C^{\infty}(\calL)$.
Because 
$\sbv{\pi^* f}{\pi^* g}=0$ 
and $\sbv{\Theta}{\Theta}=0$,
the bracket $\tbv{-}{-}$ is graded symmetric 
and satisfies the Leibniz rule and the Jacobi identity.
Throughout this section, we assume that the derived bracket 
is nondegenerate. 
Then the bracket $\tbv{-}{-}$ defines a
graded symplectic structure $\omega_{\Theta}$ on $\calL$. 

Let $\alpha_{\calL}$ be a function on $\calL$ such that 
$\alpha:= \pi^* \alpha_{\calL}$ is a canonical function 
with respect to the Lagrangian submanifold $\calL$ of 
$(T^*[n+1]\calL,\{-,-\},\Theta)$. 
By the definition of canonical functions, we have 
\begin{eqnarray}
\frac{1}{2}\{\alpha_{\calL},\alpha_{\calL}\}_{\Theta}=-\Theta|_{\calL}
-\{\Theta,\alpha\}|_{\calL}
-\frac{1}{3!}\{\{\{\Theta,\alpha\},\alpha\},\alpha\}|_{\calL}
+ \cdots.
\label{Twistedmasterequation}
\end{eqnarray}
If the right-hand side of this equation is zero,
$\tbv{\alpha_{\calL}}{\alpha_{\calL}}=0$. Then
$(\calL,\tbv{-}{-},\alpha_{\calL})$ is a QP manifold of degree $n$. 
However, in general, the master equation is violated and 
controlled by the right-hand side of Equation (\ref{Twistedmasterequation}), which
involves the homological function and the canonical function. 
This observation leads to the following definition.

\begin{definition}\label{defoftwistedQPmfd}
Let $(\calL,\{-,-\}_{\calL})$ be a $P$-manifold of degree $n$, 
and let $\alpha_{\calL}$ be a degree $n+1$ function on it. 
Then $(\calL,\{-,-\}_{\calL},\alpha_{\calL})$ is called 
a \textbf{twisted $QP$-manifold} 
if there exists a $QP$-manifold 
$(T^*[n+1]\calL,\{-,-\},\Theta)$ such that
\begin{enumerate}
\item $\{f,g\}_{\calL}=\{\{\pi^* f,\Theta\}, \pi^* g \}|_\calL$,
\item $e^{\delta_{\alpha}}\Theta|_\calL=0$,
\end{enumerate}
where $\pi: T^*[n+1]\calL \rightarrow \calL$ is the natural projection 
and $f, g \in C^{\infty}(\calL)$.
\end{definition}
In this case, we call $(T^*[n+1]\calL, \omega, \Theta)$ 
the \textbf{big QP manifold}, 
$(\calL, \sbv{-}{-}_s= \tbv{-}{-}, \alpha_{\calL})$ 
the \textbf{small (twisted) QP manifold}, 
and $(T^*[n+1]\calL, \omega, \Theta, \calL, 
\sbv{-}{-}_s, \alpha)$ 
a \textbf{QP pair}, where $\alpha = \pi^* \alpha_{\calL}$. We also say that 
$(T^*[n+1]\calL, \omega, \Theta)$ is a QP realization of $(\calL,\{-,-\}_{\calL},\alpha_{\calL})$.

The following proposition 
is an immediate consequence of the definition.
\begin{proposition}
If $(\calL,\{-,-\}_{\calL}, \alpha_{\calL})$ is a twisted QP manifold with a
QP realization $(T^*[n+1]\calL,\omega,\Theta)$, 
then
\begin{eqnarray}
\frac{1}{2}\{\alpha_{\calL},\alpha_{\calL}\}_{\calL} 
= \frac{1}{2}\{\alpha_{\calL},\alpha_{\calL}\}_{\Theta}
= - \Theta|_{\calL}
-\{\Theta,\alpha\}|_{\calL}
-\frac{1}{3!}\{\{\{\Theta,\alpha\},\alpha\},\alpha\}|_{\calL}
-\cdots.
\end{eqnarray}
\end{proposition}

In Definition \ref{defoftwistedQPmfd},
it is not clear how to decide 
whether a given $(\calL,\{-,-\}_{\calL},\alpha_{\calL})$ is 
a twisted QP manifold, 
since a twisted QP manifold needs extra data
$(T^*[n+1]\calL, \omega, \Theta)$.
Thus a QP pair is a good starting point for the analysis 
of twisted structures.
We consider this problem in Section \ref{TwistedQPAKSZ}
in terms of the deformation theory of QP manifolds and the
bulk-boundary correspondence of AKSZ sigma models.

\subsection{Examples}
\noindent
Twisted QP manifolds contain many kinds of 
twisted higher Poisson structures, for example, 
the twisted Courant algebroids, 
Nambu-Poisson structures,
and higher twisted algebroids. In this section, we show how known and new structures 
can be realized as twisted QP manifolds.

First, we give an illustrative example of
a canonical function and a twisted QP manifold that is not a QP manifold.

%
\begin{example}
\label{twistedPoissonexample}
\cite{Terashima}
Let $\calL=T^*[1]M \times \mathfrak{g}[1]$,
where $M$ is a manifold and $\mathfrak{g}$ is a quadratic
Lie algebra. To define a twisted Q-structure on $\calL$, 
we define a homological function on the shifted cotangent bundle 
$T^*[2]\calL$ with a canonical symplectic structure $\omega$.
We take local coordinates on $M$,
the fiber of $T^*[1]M$,
and $\mathfrak{g}[1]$,
$(x^i, p_i, u^a)$, respectively.
The conjugate coordinates of the fiber of $T^*[2](T^*[1]M \times \mathfrak{g}[1])$
are $(\xi_i, q^i, v_a)$.
The following function of degree $3$ is a homological function on $T^*[2]\calL$ if $\sbv{\Theta_M}{\Theta_H}=0$ and
$\sbv{\Theta_C}{\Theta_R}=0$:
\begin{eqnarray*}
\Theta &=&
\Theta_M
+
\Theta_C
+
\Theta_R
+
\Theta_H
\nonumber \\
&=&
\xi_{i} q^{i}
+ \frac{1}{2}
C_{ab}{}^{c} u^{a} u^{b} v_c
+ \frac{1}{3!}
R^{abc}v_{a} v_{b} v_{c}
+ \frac{1}{3!}
H_{ijk}(x) q^{i} q^{j} q^{k},
\end{eqnarray*}
where $C_{ab}{}^c$ 
is a Lie algebra structure constant 
and $R^{abc}$ is a constant.
Let $H$ be a $3$-form on $M$ defined
by $H=\frac{1}{3!} H_{ijk}(x) dx^{i} \wedge dx^{j} \wedge dx^{k}$,
and let $R\in\wedge^3\mathfrak{g}^*$ be defined by 
$R=\frac{1}{3!}R^{abc}v_{a} v_{b} v_{c}$.
Note that $\mathfrak{g}$, $R$ is also seen as a constant section of
$\wedge^3(M\times\mathfrak{g})$.
$\Theta$ is a homological function 
if and only if
$H$ is a closed $3$-form on $M$ and 
$R$ is a Lie algebra $3$-cocycle. 

Let $\alpha = \pi + \rho
= \frac{1}{2} \pi^{ij}(x) p_i p_j
+ \rho^{j}{}_a(x) u^a p_j
$
be a canonical function with respect to the Lagrangian submanifold
$\calL = T^*[1]M \times \mathfrak{g}[1]$, which is 
locally expressed by 
$\{\xi_i=q^i=v_a=0\}$.
From the canonical function equation,
a canonical function $\alpha$ satisfies
\[\frac{1}{2}\{\alpha_{\calL},\alpha_{\calL}\}_{\Theta}
=
-\frac{1}{3!}\{\{\{\Theta,\alpha\},\alpha\},\alpha\}|_{\calL}
\neq 0.\] 
Generally, $\alpha$ gives a Lie algebroid structure 
on $T^*M \times \mathfrak{g}$, as discussed in \cite{Lu}.

If $\Theta_C = \Theta_R = \Theta_H=0$ and $\rho=0$, the canonical
function equation
$e^{\delta_\alpha} \Theta |_{\calL} =0$ is
equivalent to 
$[\pi_{\calL}, \pi_{\calL}]_S = - \sbv{\sbv{\Theta}{\pi}}{\pi}|_{\calL} =0$, 
which defines
a nontwisted QP manifold and 
is nothing but a Poisson
structure on $M$. Here $[-, -]_S = \tbv{-}{-}
$
is identical to the Schouten bracket on $\Gamma(\wedge^{\bullet}TM)$.

If $\Theta_H=0$ and $R$ is the Cartan 3-tensor, 
then the canonical function equation
defines a quasi-Poisson structure,
$[\pi_{\calL}, \pi_{\calL}]_S = \wedge^3 \rho^{\#} R$
\cite{AKM}.

If $\rho=0$, the canonical function equation defines a
twisted-Poisson structure, $[\pi_{\calL}, \pi_{\calL}]_S = \wedge^3 \pi^{\#} H$
\cite{Klimcik:2001vg}\cite{Park2000au}\cite{Severa:2001qm}. 
The derived-derived bracket $\{\{-,\alpha_{\calL}\}_\Theta,-\}_\Theta$
gives a twisted Poisson bracket on $C^\infty(M)$.

\end{example}

\begin{example}[$H_4$-Twisted Courant algebroid]\label{twistedCourantexample}
We consider an N-manifold \\
$T^*[3]\calL = T^*[3]T^*[2]E[1]$, 
where $E \longrightarrow M$ is a vector bundle
on a manifold $M$.
We take local coordinates $(x^i, u^a, p_i)$ of degree $(0, 1, 2)$
on $T^*[2]E[1]$ and conjugate local coordinates
$(\xi_i, v_a, q^i)$ of degree $(3, 2, 1)$ on the fiber.

A graded symplectic structure is given by
$\omega = \delta x^i \wedge \delta \xi_i
- \delta u^a \wedge \delta v_a
+ \delta p_i \wedge \delta q^i$
and takes a Lagrangian submanifold $\calL = T^*[2]E[1]$, 
which is locally expressed by $\{ q^i = \xi_i = v_a  = 0 \}$.

Let us consider the following Q-structure
satisfying $\Theta |_{\calL}=0$:
\bea
\Theta &=& \xi_i q^i
+ \frac{1}{2}k^{ab} v_a v_b+\frac{1}{4!}
H_{ijkl}(x) q^{i} q^{j} q^{k} q^{l},
\eea
where $k^{ab}$ is a fiber metric on $E$.
$\Theta$ satisfies
$\sbv{\Theta}{\Theta}=0$ if and only if $dH=0$, where
$H = \frac{1}{4!} H_{ijkl}(x) dx^{i} \wedge dx^{j} \wedge dx^{k}
\wedge dx^{l}$
is a $4$-form on $M$.
This defines a Lie algebroid up to homotopy 
on $T^*E$ \cite{Ikeda:2010vz}\cite{Grutzmann}.

We take the following canonical function of degree $3$,
\bea
\alpha =
\rho^{i}{}_a(x) p_{i} u^{a}
+\frac{1}{3!}h_{abc}(x) u^a u^b u^c,
\eea
with respect to the Lagrangian submanifold
$\calL = T^*[2]E[1] 
$.
Straightforward calculations show that the canonical function 
equation $e^{\delta_{\alpha}}\Theta|_{\calL}=0$ produces
the following identities:
\begin{eqnarray}
&&
k^{ab}\rho^i{}_a\rho^j{}_b=0,\label{Courant1}
\\
&&
\frac{\partial \rho^{i}{}_a}{\partial x^j} \rho^{j}{}_b
- \frac{\partial \rho^{i}{}_b}{\partial x^j} \rho^{j}{}_a
+ k^{cd}\rho^i{}_c h_{dab}
=0,
\label{Courant2}
\\
&&
\frac{\partial h_{abc}}{\partial x^i}  \rho^{i}{}_d
+ k^{ef}h_{abe} h_{fcd}\nonumber
+ H_{ijkl}\rho^i{}_a\rho^j{}_b\rho^k{}_c\rho^l{}_d
\\ &&
\qquad
+ (abcd \ \mbox{completely skewsymmetric})=0.\label{Courant3}
\end{eqnarray}
If $H=0$, 
$e^{\delta_{\alpha}}\Theta|_{\calL}=0$ is equivalent to
$\{\alpha_{\calL},\alpha_{\calL}\}_{\Theta}=0$.
Then $(T^*[2]E[1],\{-,-\}_{\Theta},\alpha_{\calL})$
is a QP manifold of degree 2.
Equations (\ref{Courant1}), (\ref{Courant2}),
and (\ref{Courant3})
are satisfied if and only if $E$ is the Courant algebroid
\cite{lwx} \cite{Roy01}.
The Dorfman bracket of the Courant algebroid 
is given by the derived-derived bracket on
$\Gamma (E)$ by
\begin{eqnarray}
[e_1,e_2]_D= - \{\{e_1,\alpha_{\calL}\}_{\Theta},e_2\}_{\Theta},
\ \forall e_1,e_2\in\Gamma(E).
\label{CourantDorfmanbracket}
\end{eqnarray}

Let us consider the general case where $H \neq 0$.
Then the canonical function
$\alpha$ satisfies
\begin{eqnarray}
\frac{1}{2}\{\alpha_{\calL},\alpha_{\calL}\}_{\Theta}
=-\frac{1}{4!}\{\{\{\{H,\alpha\},\alpha\},\alpha\},\alpha\}
|_{T^*[2]E[1]}\neq 0,
\end{eqnarray}
from the canonical function equation.
This is a twisted QP
manifold $T^*[2]E[1]$ of degree 2.
The structure given by Equations
(\ref{Courant1}), (\ref{Courant2}), and (\ref{Courant3}) is the
$H_4$-twisted Courant algebroid \cite{Hansen:2009zd},
in which the Leibniz identity of the Dorfman bracket 
(\ref{CourantDorfmanbracket}) on $\Gamma(E)$
is broken by the $4$-form $H$
:
\begin{eqnarray}
[e_1, [e_2, e_3]_D]_D -
[[e_1, e_2]_D, e_3]_D - [e_2, [e_1, e_3]_D]_D = \wedge^4 \rho^{\#} H.
\end{eqnarray}
Note that on $M$, $H$ is a $4$-form and $h$ is a $3$-form.
Such a $4$-form twisted Courant algebroid appears naturally 
in the theory of three-dimensional topological sigma models with 
Wess-Zumino terms, 
the cotangent extension of quadratic transitive 
Lie algebroids, and the reduction of exact Courant algebroids. 
It also naturally arises from a coisotropic Cartan geometry; see \cite{Xu}.
\end{example}

\begin{example}[Twisted higher Dorfman brackets
of degree $n$]
\label{ngeneral}
Let $E \longrightarrow M$ be a vector bundle
over a manifold $M$. We take an N-manifold
$(T^*[n]T^*[n-1]E[1],\omega,\Theta)$ with a canonical function $\alpha$.  
Local coordinates on $\calL = T^*[n-1]E[1]$ are chosen as
$(x^i, u^a, p_i, w_a)$ of degree $(0, 1, n-1, n-2)$,
and those of the fiber are chosen as
$(\xi_i, v_a, q^i, z^a)$ of degree $(n, n-1, 1, 2)$.

The canonical graded symplectic structure of degree $n$ is given by
$$\omega = \delta x^i \wedge \delta \xi_i
+ (-1)^n \delta u^a \wedge \delta v_a
+ \delta p_i \wedge \delta q^i
+ (-1)^n \delta w_a \wedge \delta z^a
.$$

We define a Lagrangian submanifold by $\calL = T^*[n-1]E[1]
= \{\xi_i = v_a = q^i = z^a=0 \}$ and
a Q-structure 
satisfying $\Theta |_{\calL}=0$ as
\bea
\Theta := \xi_i q^i
+ v_a z^a
+ \frac{1}{2} C^a{}_{ij}(x) v_a q^i q^j
+ \frac{1}{(n+1)!}
H{}_{i_0 \cdots i_{n}}(x)
q^{i_0} \cdots q^{i_{n}}.
\eea
Then $\sbv{\Theta}{\Theta}=0$ produces
$d H=0$
and $d C =0$,
where
$H = \frac{1}{(n+1)!}
H{}_{i_0 \cdots i_{n}}(x)
dx^{i_0}\wedge \cdots \wedge dx^{i_{n}}$
is a $(n+1)$-form,
and $C^a =\frac{1}{2} C^a{}_{ij}(x) dx^i \wedge dx^j$
is a $2$-form taking values in $E^*$.

We take the following canonical function $\alpha$:
\bea
\alpha =
\rho^{i}{}_a(x) u^{a} p_{i}
+ \frac{1}{2} f^a{}_{bc}(x) w_{a} u^{b} u^{c}
+ \frac{1}{n!} h_{a_1 \cdots a_n}(x)u^{a_1} \cdots u^{a_n}.
\eea
Direct calculation shows that 
$e^\delta_{\alpha}\Theta|_{\calL}=0$ produces the equations:
\begin{eqnarray}
&& \frac{\partial \rho^{i}{}_a}{\partial x^j} \rho^{j}{}_b
- \frac{\partial \rho^{i}{}_b}{\partial x^j} \rho^{j}{}_a
+ \rho^i{}_c f^c{}_{ab}
- C^a{}_{kl} \rho^{i}{}_c \rho^{k}{}_a \rho^{l}{}_b
=0,\label{higher1}
\\
&&
\frac{\partial f^a{}_{bc}}{\partial x^i}  \rho^{i}{}_d
+ f^a{}_{be} f^e{}_{cd}
- C^e{}_{ij} f^a{}_{be} \rho^{i}{}_c \rho^{j}{}_d
+ (bcd \ \mbox{completely skewsymmetric})=0,\label{higher2}
\\
&&
\frac{\partial h_{a_0 \cdots a_{n-1}}}{\partial x^i} \rho^i{}_{a_{n}}
+ h_{e a_2 \cdots a_{n}} f^e{}_{a_{0} a_{1}}
+ C^e{}_{ij} h_{e a_0\cdots a_{n-2}} \rho^i{}_{a_{n-1}} \rho^j{}_{a_{n}}
+
H{}_{j_0 \cdots j_{n}}(x)
\rho^{j_1}{}_{a_{0}} \cdots \rho^{j_{n}}{}_{a_{n}}
\nonumber \\
&&
+ (a_0\cdots a_{n} \ \mbox{completely skewsymmetric})
=0.\label{higher3}
\end{eqnarray}
Note that the derived bracket $\tbv{-}{-}$ is just the canonical
nondegenerate Poisson bracket on $\calL = T^*[n-1]E[1]$.
Since $\Gamma (E \otimes \wedge^{n-2} E^*)$ is identified as a 
subspace of 
$C^{\infty}(T^*[n]E[1])$,
we can define a bracket on
$\Gamma (E \otimes \wedge^{n-2} E^*)$
by the derived-derived bracket:
\begin{eqnarray}
[-,-]_D = - \tbv{\tbv{-}{\alpha_{\calL}}}{-}.
\end{eqnarray}
If $C=H=0$, we have $\tbv{\alpha_{\calL}}{\alpha_{\calL}}=0$.
From Equations (\ref{higher1}), (\ref{higher2}), and (\ref{higher3}), $[-,-]_D$
is just the higher Dorfman bracket on $E \otimes \wedge^{n-2} E^*$.
If $C$ or $H$ is nonzero,
$(T^*[n-1]E[1],\tbv{-}{-},\alpha_{\calL})$ 
is a twisted QP manifold.
In particular, if $C=0$ and $H$ is nonzero, 
we obtain a twisted higher Dorfman bracket 
where the Leibniz identity of the Dorfman bracket is broken by
a closed $n+1$-form $H$.
\end{example}
\begin{example}
A \textbf{Nambu-Poisson bracket} of order $n$ $(\geq 3)$
on $M$ is a skewsymmetric
linear map
$\{\cdot, \cdots , \cdot \}:
C^{\infty}(M)^{\otimes n} \longrightarrow C^{\infty}(M)$ such that
\begin{eqnarray}
&& (1) \quad \{f_{\sigma(1)}, f_{\sigma(2)}, \cdots , f_{\sigma(n)} \}
= (-1)^{\epsilon(\sigma)} \{f_1, f_2, \cdots , f_n \},
\nonumber \\
&& (2) \quad \{f_1 g_1, f_2, \cdots , f_n \}
= f_1 \{g_1, f_2, \cdots , f_n \}
+ g_1 \{f_1, f_2, \cdots , f_n \},
\nonumber \\
&& (3) \quad \{f_1, f_2, \cdots, f_{n-1}, \{ g_1, g_2,
\cdots, g_{n} \} \}
\nonumber \\
&&
= \sum_{k=1}^n
\{g_1, \cdots, g_k,
\{ f_1, f_2, \cdots, f_{n-1}, g_k \},
g_{k+1}, \cdots, g_{n} \}.
\nonumber
\end{eqnarray}

The $n$-vector field $\pi\in\Gamma( \wedge^n TM)$ is called
the Nambu-Poisson tensor field, and it is defined as
$$\pi(df_1, df_2, \cdots , df_n ) = \{f_1, f_2, \cdots , f_n \}.$$

Let us assume that the Nambu-Poisson tensor is decomposable.
Then a canonical function for the Nambu-Poisson structure is constructed 
as a special case of Example \ref{ngeneral},
as follows. 
Consider an N-manifold
$T^*[n]\calL = T^*[n]T^*[n-1]E[1]$, where $E= \wedge^{n-1} T^*M$.
We take local coordinates $(x^i, v_I, p_i, w^I)$ of degree $(0,1,n-1,n-2)$ 
on $T^*[n-1]E[1]$ and conjugate local coordinates
$(\xi_i, u^I, q^i, z_I)$ of degree $(n, n-1, 1, 2)$ on the fiber,
where $I$ is the multi-index $I = (i_1, i_2, \cdots, i_{n-1})$.

The canonical graded symplectic structure of degree $n$ can be expressed
as
$$\omega = \delta x^i \wedge \delta \xi_i
+ (-1)^n \delta v_I \wedge \delta u^I
+ \delta p_i \wedge \delta q^i
+ (-1)^n \delta w^I \wedge \delta z_I.
$$

A Q-structure function $\Theta$ is defined as
$$\Theta = - q^i \xi_i
+ \frac{1}{(n-1)!} z_I (u^I- q^{i_1} \cdots q^{i_{n-1}}),
$$
which automatically satisfies $\sbv{\Theta}{\Theta}=0$.
$\Theta$ defines the Dorfman bracket on $TM \oplus \wedge^{n-1} T^*M$
\cite{Dorfman}
by the derived bracket $[-,-]_D= - \sbv{\sbv{\Theta}{-}}{-}$.
We take the function $\alpha$ to be
\begin{eqnarray}
\alpha = - \frac{1}{(n-1)!} \pi^{i_1 \cdots i_{n-1} i_n}(x)
v_{i_1 \cdots i_{n-1}} p_{i_n}.
\end{eqnarray}
\begin{proposition}
Let $\calL=T^*[n-1]E[1]$ be the Lagrangian submanifold of $T^*[n]\calL$.
Then $\alpha$ is a canonical function with respect to $\Theta$ and $\calL$, 
i.e., $e^{\delta_\alpha} \Theta|_{\calL}=0$
if and only if $\pi$ is a decomposable Nambu-Poisson tensor.
{\rm \cite{Bouwknegt:2011vn}}
\end{proposition}
\end{example}

\subsection{Strong Courant Algebroids
}\label{sectionSCA}
\noindent
Let $E$ and $A$ be two vector bundles on $M$.
We consider an N-manifold
$T^*[3] \calL = T^*[3] (E[1] \oplus A[2])$.
Let us take local coordinates $(x^i, u^a, z_p)$ on $M$,
$E[1]$ and $A[2]$,
and local coordinates $(\xi_i, v_a, w^p)$ 
of the fiber of $T^*[3]$ of degree $(3, 2, 1)$, 
respectively.

A canonical graded symplectic structure is
$\omega = \delta x^i \wedge \delta \xi_i
- \delta u^a \wedge \delta v_a
+ \delta w^p \wedge \delta z_p$.
%
We define a Q-structure function as
\begin{eqnarray}
\Theta
&=&
\frac{1}{2}
k^{ab} v_{a} v_{b}
+ \rho^i{}_r(x) \xi_{i} w^{r}
+ \frac{1}{2}
C^r{}_{pq}(x) z_r w^p w^q,
\end{eqnarray}
where $k^{ab}$ is a fiber metric on $E$.
$\sbv{\Theta}{\Theta}=0$ is equivalent to the following identities:
\begin{eqnarray}
&& \rho^i{}_r \frac{\partial \rho^i{}_s}{\partial x^j}
- \rho^i{}_s \frac{\partial \rho^i{}_r}{\partial x^j}
- \rho^i{}_p C^p{}_{rs}=0,
\nonumber \\&&
- \rho^i{}_p \frac{\partial C^s{}_{qr}}{\partial x^i}
+ C^s{}_{pt} C^t{}^{qr} + (pqr \ \mbox{cyclic})=0.
\end{eqnarray}
This condition is satisfied if $(A, C, \rho)$ is a Lie algebroid,
where 
a Lie bracket is $[e_p, e_q]_A = C^r{}^{pq}(x) e_r$ and
$\rho$ is a bundle map from $A$ to $TM$ defined by 
$\rho(e_r) = \rho^i{}_r(x) \frac{\partial}{\partial x^i}$
for $e_r \in \Gamma A$.

Let the Lagrangian submanifold
${\calL} \subset T^*[3] \calL$ be
$\{
\xi_i = v_a  =
w^p = 0 \},$
and let a canonical function of degree $3$ with 
respect to the Lagrangian submanifold ${\calL}$
be
\begin{eqnarray*}
\alpha
=
\tau^r{}_{a}(x) z_r u^a
+ \frac{1}{3!} h_{abc}(x) u^a u^b u^c.
\end{eqnarray*}
The canonical function equation $e^{\delta_\alpha} \Theta |_{\calL}=0$ 
is equivalent to
$\Theta |_{{\calL}^{\prime}}=0$,
where ${\calL}^{\prime}$ is a Lagrangian submanifold
with respect to the inverse canonical transformation of the P-structure,
$\omega^{\prime} = - d(e^{-\delta_\alpha} \vartheta)$.
${\calL}^{\prime}$
is defined by
\begin{eqnarray}
\xi_i &=& \sbv{\xi_i}{\alpha}
=
- \frac{\partial \tau^r{}_{a}}{\partial x^i}(x) z_r u^a
- \frac{1}{3!} \frac{\partial h_{abc}}{\partial x^i}(x) u^a u^b u^c,
\nonumber \\
v_a &=& \sbv{v_a}{\alpha}
=
- \tau^r{}_{a}(x) z_r
- \frac{1}{2} h_{abc}(x) u^b u^c,
\nonumber \\
w^r &=& \sbv{w^r}{\alpha}
= \tau^r{}_{a}(x) u^a.
\end{eqnarray}
Substituting this equation into $\Theta |_{{\calL}^{\prime}}=0$, 
we obtain the conditions for the canonical function, 
as follows:
\begin{eqnarray}
&& k^{ab} \tau^r{}_a \tau^s{}_b =0,
\label{StrongCourant1}
\\
&&
k^{cd} \tau^r{}_c h_{dab}
+ \rho^{i}{}_s \tau^s{}_a \frac{\partial \tau^{r}{}_b}{\partial x^i}
- \rho^{i}{}_s \tau^s{}_b \frac{\partial \tau^{r}{}_a}{\partial x^i}
+ C^r{}_{pq} \tau^{p}{}_a \tau^{q}{}_b
=0,
\label{StrongCourant2}
\\
&&
\rho^i{}_{r} \tau^r{}_{d}
\frac{\partial h_{abc}}{\partial x^i}
- \frac{1}{2} k^{ef} h_{eab} h_{fcd}
+ (abcd \ \mbox{complete skewsymmetric})
=0.
\label{StrongCourant3}
\end{eqnarray}

Now let us analyze the geometric structure
defined by this canonical function.
Let $\tau:E\longrightarrow A$ be the bundle map defined 
by the above $\tau^r{}_d$ on local charts.
Moreover, we define the following operations
by the derived-derived bracket:
\begin{eqnarray}
[e_1,e_2]_D&=&-\{\{e_1,\alpha_{\calL}\}_{\Theta},e_2\}_{\Theta},\\
\langle e_1,e_2\rangle&=&\{e_1,e_2\}_{\Theta},\\
\calD(f)&=&\{\alpha_{\calL},f\}_{\Theta},
\end{eqnarray}
where $\{-,-\}_{\Theta}:=-\{\{-,\Theta\},-\}|_{\calL}$ is the derived bracket, 
$[-,-]_D$ is a bilinear bracket, $\langle-,-\rangle$ is 
an inner product on $\Gamma(E)$, and $\calD$ 
is a map from $C^\infty(M)$ to $\Gamma(E)$.

By Equations (\ref{StrongCourant1}), (\ref{StrongCourant2}), and
(\ref{StrongCourant3}) for $e_1,e_2,e_3\in\Gamma(E)$ and
$\xi_1,\xi_2\in\Gamma(A^*)$, we have
\\
$(a)$ $\tau[e_1,e_2]_D=[\tau e_1,\tau e_2]_A$,
\\
$(b)$ $\langle\tau^*(\xi_1),\tau^*(\xi_2)\rangle=0$,
\\
$(c)$ $[e_1,e_1]_D=\calD\langle
e_1,e_1\rangle=(\rho\circ\tau)^*d\langle e_1,e_1\rangle$,
\\
$(d)$ $[e_1,[e_2,e_3]_D]_D=[[e_1,e_2]_D,e_3]_D+[e_2,[e_1,e_3]_D]_D$,
\\
$(e)$ $\rho\circ\tau(e_1)\langle e_2,e_3\rangle=\langle [e_1,e_2]_D,e_3\rangle+\langle e_2,[e_1,e_3]_D\rangle$.

We call this structure 
a strong Courant algebroid.
\begin{definition}
A \textbf{strong Courant algebroid} is
$(E,\tau,\langle -, - \rangle, [-,-]_D,A,\rho,[-,-]_A)$
satisfying the equation $(a)-(e)$, where $E$ is a vector bundle on
a manifold $M$, $\langle -, - \rangle$ is an inner product on $E$,
$[-,-]_D$ is a bilinear operator, $(A,\rho, [-,-]_A)$ is a Lie
algebroid on M, and $\tau:E\longrightarrow A$ is a bundle map.
\end{definition}
Roughly speaking, a strong Courant algebroid is a Courant algebroid over a Lie algebroid. Obviously, any Courant algebroid is naturally a strong Courant algebroid.
\begin{proposition}
Let $(E,\tau,\langle -, - \rangle, [-,-]_D,A,\rho,[-,-]_A)$
be a strong Courant algebroid.
Then
$(E,\tau \circ \rho,\langle -, - \rangle, [-,-]_D)$
is a Courant algebroid.
\end{proposition}
Thus for any strong Courant algebroid,
there exists a Courant algebroid associated with it.
\begin{example}
If $M$ is a point, the strong Courant algebroid is a triple 
$(\mathfrak{g}_1, \mathfrak{g}_2,\tau)$, where $\mathfrak{g}_1$ 
is a quadratic Lie algebra, $\mathfrak{g}_2$ is a Lie algebra, and 
$\tau$ is a homomorphism from $\mathfrak{g}_1$ to $\mathfrak{g}_2$ 
such that $\tau^*(\mathfrak{g}_2^*)$ is an isotropic 
subspace of $\mathfrak{g}_1$. 
Any Lie bialgebra is an example of the strong Courant algebroid.
\end{example}
\begin{example}
Let $A$ be a Lie algebroid. An inner product on $A \oplus A^*$ 
is defined by the natural pairing of $A$ and $A^*$, 
an anchor map $\tau$ is defined by the natural projection 
from $A \oplus A^*$ to $A$. 
It is clear that $(A \oplus A^*,\tau,A,[-,-]_D)$ gives a 
strong Courant algebroid, where
$[-,-]_D$ is the Dorfman bracket given by
\[[X+\xi,Y+\eta]_D=[X,Y]_A+L_X\eta-i_Yd\xi,\]
$X,Y\in\Gamma(A)$, $\xi,\eta\in\Gamma(A^*)$,
and $L$ and $d$ are the Lie derivative and the de Rham differential
associated to $A$, respectively.
\end{example}
\begin{example}
Let $P$ be a $G$-principal bundle over $M$.
Since $G$ acts on $TP \oplus T^*P$ naturally, we get a bundle 
$\frac{TP \oplus T^*P}{G}$ over $M$ by reduction, 
where the $G$-invariant sections of
$TP \oplus T^*P$ reduce to the sections of
the bundle $\frac{TP \oplus T^*P}{G}$.

We define a bundle map $\tau$ by the natural projection 
from $\frac{TP \oplus T^*P}{G}$
to the Atiyah algebroid $\frac{TP}{G}$.
Because of the $G$-invariance,
the canonical Dorfman bracket and the natural pairing on $TP \oplus T^*P$ 
induce a bracket $[-,-]_D$ and an inner product
$\langle-,-\rangle$ on $\Gamma(\frac{TP \oplus T^*P}{G})$.
It is easy to confirm that
$\left(\frac{TP \oplus T^*P}{G},\tau,\langle-,-\rangle,[-,-]_D,\frac{TP}{G}
\right)$
is a strong Courant algebroid.
\end{example}

Following the discussion in Subsection \ref{twistedQPmfd}, 
we can easily generalize the strong Courant algebroid 
to the twisted version.
For most of the concepts that appear in the Courant
algebroids, parallels can be introduced in the strong Courant algebroids. 
The A-connections and morphisms between Lie algebroids
\cite{Fernandes} will play a key role in the further study of the strong 
Courant algebroids.

\section{Twisted QP Manifolds and bulk-boundary AKSZ sigma models
}\label{TwistedQPAKSZ}
\subsection{QP Manifolds as QP Pairs}
\label{fromQPtoQPpair}
\noindent
As a physical application, we will show how a QP pair 
encodes the bulk-boundary correspondence
of AKSZ sigma models.

The following theorem justifies the view that a twisted QP manifold 
is a generalization of a QP manifold.
%
\begin{theorem}\label{derivedQPtheorem}
A QP manifold is a twisted QP manifold.
\end{theorem}
\begin{proof}
Let $(\calL,\omega_{\calL}, \alpha_{\calL})$ be a QP manifold of degree $n$.
We consider the Poisson bracket $\sbv{-}{-}_{\calL}$ for 
the symplectic structure $\omega_{\calL}$.
We can construct a Poisson bivector field $\pi_{\calL}$
to define the Poisson bracket 
$\sbv{f}{g}_{\calL} = \pi_{\calL}(\delta f, \delta g)$, where
$f, g \in C^{\infty}(\calL)$
and $\delta$ is a differential on $\calL$.

Consider a shifted cotangent bundle $T^*[n+1]\calL$,
and choose a canonical graded symplectic structure $\omega = \omega_{can}$
on $T^*[n+1]\calL$. 
%
From the theory of supergeometry, we have that
a bivector field on $\calL$ can be identified as a function
on $T^*[n+1] \calL$.
Therefore we obtain $\Theta \in C^{\infty}(T^*[n+1] \calL)$
corresponding to $\pi_{\calL}$.
We can easily prove that $\pi_{\calL}(\delta f, \delta g) 
= \sbv{\sbv{\pi^* f}{\Theta}_{can}}{\pi^* g}_{can}|_{\calL}$,
which is similar to the formula for the usual Poisson bracket.
Therefore the small Poisson bracket $\sbv{-}{-}_{\calL}$ is rederived from the 
derived bracket of the big bracket.

Note that $\alpha = \pi^* \alpha_{\calL}$ is a canonical function
for this $\Theta$ because 
$\sbv{\alpha_{\calL}}{\alpha_{\calL}}_{\calL}
= \sbv{\sbv{\alpha}{\Theta}_{can}}{\alpha}_{can}|_{\calL}=0$.
Therefore, 
$(T^*[n+1]\calL, \omega_{can}, \Theta, \alpha)$ is
a QP pair, and
the original QP manifold $\calL$ is its small QP manifold.


Let us describe it using a local coordinate. 
Let $q^i$ be a local coordinate of degree $|q^i|$ on $\calL$
such that $\sbv{q^i}{q^j}_{\calL}= (\omega_{\calL}^{-1})^{ij}(q)$.
There exists a local Darboux coordinate $p_i$ of degree $|p_i|$ on the fiber of
$T^*[n+1]\calL$ such that 
$$\sbv{q^i}{p_j}_{can} = -(-1)^{(|q^i|-n-1)(|p_j|-n-1)} 
\sbv{p_j}{q^i}_{can} = \delta^i{}_j,$$
where $\sbv{-}{-}_{can}$ is the Poisson bracket defined from the 
canonical graded symplectic form $\omega_{can}$ on $T^*[n+1]\calL$.
If we define 
$\Theta = -(-1)^{(|q^i|-n-1)(|p_j|-n-1)}
\frac{1}{2} (\omega_{\calL}^{-1})^{ij}(q) p_i p_j$, 
then a straightforward computation gives
$\sbv{\Theta}{\Theta}_{can}=0$ and
$\sbv{-}{-}_{\calL} = \{\{-, \Theta \}_{can},-\}_{can}|_{\calL}$.

This construction shows that all the terms 
$\Theta|_{\calL}$,
$\{\Theta,\alpha\}_{can}|_{\calL}$,
$\{\{\Theta,\alpha\}_{can},\alpha\}_{can}|_{\calL}$, \\
$\{\{\{\Theta,\alpha\}_{can},\alpha\}_{can},\alpha\}_{can}|_{\calL}, \cdots$,
are zero 
for the canonical $\Theta$.
Therefore, we obtain
\begin{eqnarray}
e^{\delta_\alpha}\Theta|_\calL
&=&\Theta|_{\calL}
+\{\Theta,\alpha\}_{can}|_{\calL}
+\frac{1}{2}\{\{\Theta,\alpha\}_{can},\alpha\}_{can}|_{\calL}
+\frac{1}{3!}\{\{\{\Theta,\alpha\}_{can},\alpha\}_{can},\alpha\}_{can}|_{\calL}+\cdots
\nonumber \\
&=&0.
\end{eqnarray}
Thus, $\alpha$ is a canonical function. \qed
\end{proof}
From the proof of this theorem, there exists a
QP pair for any QP manifold.

\subsection{Twisted QP Manifolds from Deformations}
\label{QPdeformationtheory}
\noindent
In this subsection, we show how twisted QP structures 
on $\calL$ can be obtained by the deformation
of the canonical $\Theta$ on $T^*[n+1]\calL$. 
As an application, we can add ``fluxes'' 
to $\calL$ by deforming the corresponding 
canonical homological function on $T^*[n+1]\calL$.
This also leads to a common method 
for constructing a twisted QP manifold.

First, we take a P-manifold $(\calL, \omega_{\calL})$ of degree $n$.
From the proof of Theorem \ref{derivedQPtheorem}, there exists
a QP realization $(T^*[n+1]\calL,\{-,-\},\Theta)$
with $\alpha=0$, where
$\Theta$ is the canonical Q-structure.
By definition, $\tbv{-}{-} = \sbv{\sbv{-}{\Theta}}{-}|_{\calL}$
coincides with the original Poisson bracket derived from $\omega_{\calL}$. 

Next, we consider a deformation of $\Theta$, $\Theta_d = \Theta +
\Theta^{\prime}$, such that $\sbv{\Theta_d}{\Theta_d}=0$. 
This can be rewritten as the Maurer-Cartan equation,
$\{\Theta,\Theta^{\prime}\} + \frac{1}{2}
\sbv{\Theta^{\prime}}{\Theta^{\prime}} =0$.
For simplicity, we concentrate on the deformation in which
the derived P-structure on $\calL$ is not changed,
i.e., $\{\{-,\Theta_d \},-\}|_{\calL}
=\{\{-,\Theta\},-\}|_{\calL}$.
One such solution is the deformation for which 
$\Theta^{\prime}|_{\calL^{\perp}}=0$.

A function $\alpha$ of degree $n+1$ on $\calL$
is a canonical function on $(T^*[n+1]\calL, \{-,-\},\Theta_d)$ 
with respect to $\calL$ if and only if
\begin{eqnarray}\label{deformedtheta}
e^{\delta_{\alpha}} \Theta_d |_{\calL} =
e^{\delta_{\alpha}} (\Theta + \Theta^{\prime}) |_{\calL} = 0.
\end{eqnarray}
%
From the above assumption, 
the equation $e^{\delta_{\alpha}}\Theta|_{\calL}=0$ 
is equivalent to the master equation 
$\{\alpha_{\calL},\alpha_{\calL}\}_{\Theta}=0$
for the canonical homological function $\Theta$ on $T^*[n+1]\calL$. 
However, the equation \eqref{deformedtheta} for $\Theta_d$ generally 
breaks the classical master equation.
To see this, we recall that
$e^{\delta_{\alpha}} \Theta_d |_{\calL} =0$ produces
the equation
\[
\frac{1}{2}\{\alpha_{\calL},\alpha_{\calL}\}_{\Theta_d}= -
(\Theta+\Theta')|_{\calL} -\{\Theta+\Theta',\alpha\}|_{\calL}
-\frac{1}{3!}\{\{\{\Theta+\Theta',\alpha\},\alpha\},\alpha\}|_{\calL} \cdots.
\]
Note that for a canonical $\Theta$, the terms
$\Theta|_{\calL}$ and
$\{\{...\{\Theta,\alpha\},...,\alpha\},\alpha\}|_{\calL}$ on the
right-hand side are equal to zero.
However, the terms involving a higher bracket between 
$\Theta'$ and $\alpha$, for example,
$\{\{\{\Theta^{\prime}, \alpha\},\alpha\},\alpha\}|_{\calL}$, may be nonzero.
Also, $\{\alpha_{\calL},\alpha_{\calL}\}_{\Theta_d}$ 
is generally nonzero. 
Therefore, a solution $\alpha$ of the canonical function of Equation 
\eqref{deformedtheta}
with respect to 
$\Theta_d$ 
gives a twisted QP manifold
$(\calL, \sbv{-}{-}_{{\Theta}_d}, \alpha_{\calL})$.
The following general proposition is obtained.
\begin{proposition}
Let $(\calL,\omega_{\calL})$
be a P-manifold, and let $(T^*[n+1]\calL,\omega, \Theta)$ be
the canonical QP realization given in Theorem \ref{derivedQPtheorem}.
Then there is a twisted QP manifold
$(\calL,\omega_{\calL}, \alpha_{\calL})$
associated with any function $\Theta^{\prime}$ of degree $n+2$ 
on $T^*[n+1]\calL$ satisfying
\begin{enumerate}
\item $\{\Theta,\Theta^{\prime}\} 
+ \frac{1}{2} \sbv{\Theta^{\prime}}{\Theta^{\prime}} =0$;
\item
$\{\{-,\Theta' \},-\}|_{\calL}=0$;
\item
$
e^{\delta_{\alpha}} \Theta' |_{\calL}=0$,
\end{enumerate}
where $\alpha = \pi^* \alpha_{\calL}$.
In this case, we say that the twisted QP 
manifold $(\calL,\omega, \alpha_{\calL})$ is twisted by $\Theta'$.
\end{proposition}

One can find an equivalent description of a twisted QP manifold 
in the frame of homotopy Poisson manifolds in \cite{xiaomeng}, 
where the authors encode deformations of the canonical homological function 
on $(T^*[n+1]\calL,\omega)$ by higher derived brackets. 
A deformation of the homological function on $T^*[n+1]\calL$ 
induces a homotopy Poisson algebra on $C^{\infty}(\calL)$, 
and the canonical transformation equation $e^{\delta_\alpha}\Theta|_{\calL}=0$
becomes the Maurer-Cartan equation for the underlying $L_{\infty}$-algebra, 
which can be seen as a homotopy version of the normal master equation. 
As a result, a twisted QP manifold can be viewed as a homotopy 
version of a QP manifold.

\subsection{Bulk-Boundary Correspondence of AKSZ Sigma Models}
\label{bulkboundaryAKSZ}
\noindent
Generally, a kind of (topological) quantum field theory in $n+1$
dimensions on $X$
has the same physical constants as
a quantum field theory in $n$ dimensions on 
the boundary $\partial X$.
Theorem \ref{derivedQPtheorem} says that
there is another QP manifold associated 
with any (twisted) QP manifold.
Therefore, we obtain a pair of AKSZ sigma models 
associated with one QP manifold.
This is can be interpreted physically as the
bulk-boundary correspondence (holographic correspondence) 
of physical models.


Let us take a QP pair
$(T^*[n+1]\calL, \omega, \Theta, \calL, \alpha)$.
First, we suppose that 
$\alpha$ satisfies $\tbv{\alpha_{\calL}}{\alpha_{\calL}}=
\sbv{\sbv{\alpha}{\Theta}}{\alpha}|_{\calL} =0$,
i.e., $(\calL, \tbv{-}{-}, \alpha_{\calL})$
is a nontwisted QP manifold, where $\alpha = \pi^* \alpha_{\calL}$
with $\pi: T^*[n+1]\calL \longrightarrow \calL$.

We obtain two AKSZ sigma models from these data.
Let $X$ be a manifold in $n+1$ dimensions with boundaries.
By the AKSZ construction, the original big QP manifold
$(T^*[n+1]\calL, \omega, \Theta, \calL, \alpha)$ 
defines an AKSZ sigma model with a boundary 
on $\calX = T[1]X$.
A P-structure is given by
$\bomega = \mu_* \ev^* \omega$, and
a Q-structure by (\ref{akszwithboundary}) 
or, equivalently, by (\ref{ChangeS}).

On the other hand, with the AKSZ construction,
the small QP manifold
$(\calL, \tbv{-}{-}, \alpha_{\calL})$
defines an AKSZ sigma model on $\partial \calX = T[1]\partial X$.
The Q-structure is 
\begin{eqnarray}
S_{\calL} &=& S_{\calL 0} + S_{\calL 1}
\nonumber \\
&=& \iota_{\hat{D}_{\partial X}}
\mu_{\partial \calX *} {\rm ev}^*
\vartheta_{\Theta}
+ \mu_{\partial \calX *} 
(i_{\partial} \times {\rm id})^*
\ev^* \, 
\alpha,
\label{akszonboundary}
\end{eqnarray}
where
$\omega_{\Theta} = - \delta \vartheta_{\Theta}$.

These two AKSZ sigma models, which are constructed from one QP pair, 
have a canonical correspondence.
Physical arguments show that the 
BV actions $S$ and $S_{\calL}$ correspond.
First we consider the bulk Q-structure function 
(\ref{akszwithboundary}) on $\calX$.
The Q-structure (\ref{akszonboundary}) 
on the Lagrangian submanifold $\calL$
is obtained by using Stokes' theorem
and
integrating out the auxiliary superfields.
(i.e., $S$ and $S_{\calL}$ derive the same orbits 
of the equations of motion on $\partial \calX$.)
The boundary P-structure is obtained as
$\omega_{\Theta} = - \delta \vartheta_{\Theta}$
from the first term 
$\iota_{\hat{D}_{\partial X}}
\mu_{\partial \calX *} {\rm ev}^*
\vartheta_{\Theta}$
in the resulting Q-structure $S_{\calL}$.
This AKSZ sigma model on $\partial\calX$
is equivalent to the one constructed from the small QP manifold.

The field configurations for the two AKSZ sigma models 
satisfy the following commutative diagram:
$$
\begin{CD}
\calX @> \phi_{\calX} >> T^*[n+1]\calL \\
@A i_{\calX} AA @A i_{\calL} AA \\
\partial \calX @> \phi_{\partial\calX} >> \calL
\end{CD}
$$
where
$\phi_{\calX} \in \Map(\calX, T^*[n+1]\calL)$,
$\phi_{\partial\calX} \in \Map(\partial\calX, \calL)$,
$i_{\calX}$ is the inclusion of the boundary into the manifold,
and $i_{\calL}$ is the inclusion as the zero section
such that $\pi \circ i_{\calL} = \textrm{id}_{\calL}$.
This demonstrates the bulk-boundary correspondence
of AKSZ sigma models.


We denote 
a QP pair
$(T^*[n+1]\calL, \omega, \Theta, \calL, \alpha)$
by $\QPM_{T^*[n+1]\calL}$
and a QP manifold without a canonical function
$(\calL, \omega_s, \alpha)$
by $\QPM^0_{\calL}$.
Let $\AKSZ_{\calX, \calL}$ be an AKSZ sigma model
on $\Map(\calX, \calL)$, and let
 $\AKSZ^0_{\calX, \calL}$ be an AKSZ sigma model
without a boundary (without a canonical function).
Schematically, the following diagram
shows two procedures:
$$
\begin{CD}
\QPM_{T^*[n+1]\calL} @> >> \AKSZ_{\calX, T^*[n+1]\calL} \\
@V  VV @V  VV \\
\QPM^0_{\calL} @>  >> \AKSZ^0_{\partial \calX, \calL}
\end{CD}
$$
Thus, these two methods 
produce the same AKSZ sigma model $\AKSZ^0_{\partial \calX, \calL}$.
In one method, we construct an AKSZ sigma model with a boundary 
and then reduce the theory to the boundary.
In the other method, we reduce a QP pair 
to a small QP manifold and then construct an AKSZ sigma model.

Conversely, we can take a QP manifold of degree $n$,
$(\calL, \sbv{-}{-}_{\calL}, \alpha_{\calL})$.
Following Theorem \ref{derivedQPtheorem}, we have a canonical lift
$(T^*[n+1]\calL, \omega, \Theta, \calL, \alpha)$.
Although in general, a big QP manifold is not unique to 
one small QP manifold,
we can show that for any lift
$(T^*[n+1]\calL, \omega, \Theta, \calL, \alpha)$,
the resulting small AKSZ sigma models on $\calL$
have the equivalent QP structures.
The equivalence relation for $\Theta$ for a fixed $\omega$
has the following physical interpretation in the AKSZ sigma models.
Two Q-structure functions are equivalent, 
$S^{\prime}_{T^*[n+1]\calL} \sim S_{T^*[n+1]\calL}$,
if 
$S^{\prime}_{T^*[n+1]\calL} = S_{T^*[n+1]\calL} + S_{{\sc bulk}}$
on $\calX$ have the same boundary conditions, 
that is, $S_{{\sc bulk}}$ has the boundary condition such that 
$S_{{\sc bulk}}|_{\partial \calX}=0$.


\subsection{Bulk-Boundary Correspondence of Twisted AKSZ Sigma Models
induced from General QP Pairs}
\noindent
In this subsection, we consider general cases, i.e., 
$\tbv{\alpha_{\calL}}{\alpha_{\calL}} \neq 0$.
Given a QP pair,
$(T^*[n+1]\calL, \omega, \Theta, \calL, \alpha)$, then
by definition, $(\calL, \sbv{-}{-}_{\calL}, \alpha_{\calL})$ 
is a twisted QP manifold, 
where $\pi: T^*[n+1]\calL \longrightarrow \calL$ is a 
natural projection
and $\alpha = \pi^* \alpha_{\calL}$.
%

Let us consider the AKSZ sigma models.
Let $X$ be a manifold in $n+1$ dimensions with a boundary.
By the AKSZ construction in Subsection \ref{sectionAKSZwboundary},
there is an AKSZ sigma model with a boundary from 
$\calX=T[1]X$ 
to $T^*[n+1]\calL$.
The Q-structure function is given 
in Equation (\ref{akszwithboundary}).
A boundary theory can be constructed by the method
inspired by physical arguments. 
The field configurations for two AKSZ sigma models 
are required to satisfy the following commutative diagram:
$$
\begin{CD}
\calX @> \phi_{\calX} >> T^*[n+1]\calL \\
@A i_{\calX} AA @A i_{\calL} AA \\
\partial \calX @> \phi_{\partial\calX} >> \calL
\end{CD}
$$
where $\pi \circ i_{\calL} = \textrm{id}_{\calL}$.
We restrict the equations of motion of the big AKSZ sigma model on $\calX$
to the boundary $\partial\calX$.
The derived bracket P-structure and
the Q-structure function
on the Lagrangian submanifold $\Map(\partial \calX, \calL)$ 
are constructed by using Stokes' theorem 
and integrating out 
the auxiliary superfields from the bulk Q-function $S$, in a manner
similar to that for the AKSZ sigma model cases in Subsection 
\ref{bulkboundaryAKSZ}
\cite{Park2000au}\cite{Bouwknegt:2011vn}.

The P-structure on $\Map(\partial \calX, \calL)$ is 
the same as the AKSZ construction on a small P-manifold.
The Poisson bracket derived from 
$\bomega_{\calL} = \mu_{\calL *} \ev^* \omega_{\calL}$ 
on $\calL$
is equal to the derived bracket
$
\sbv{\sbv{-}{-}}{-}|_{\calL}$
with respect to $\omega_{\calL}$.

The bulk Q-structure is expressed by 
Equation (\ref{akszwithboundary}), and
we use the equivalent Q-structure function (\ref{ChangeS}).
By restricting $\Map(\calX, \calM)$ 
to orbits of solutions in the equations of motion 
with respect to the compliment space $\calL^{\perp} = T^*[n+1]\calL/\calL$
by solving $\delta_{\calL^{\perp}} S^{\prime\prime}=0$
and using Stokes' theorem,
the bulk Q-structure $S^{\prime\prime}$
reduces to the boundary twisted Q-structure $S_{\calL}$.
We will demonstrate this procedure in Example \ref{twistedPSM}.
$S_{\calL}$ has extra terms that are not generated
by the original AKSZ construction on $\Map(T[1] \partial X, \calL)$.
A direct calculation shows that the boundary Q-structure function
is the following untwisted Q-structure function 
plus extra terms:
\begin{eqnarray}
S_{\calL} &=& S_{\calL, \partial\calX} + S_{\calL, \calX}
\nonumber \\
&=& \iota_{\hat{D}_{\partial X}}
\mu_{\partial \calX *} {\rm ev}^*
\vartheta_{\Theta}
- \mu_{\partial \calX *} (i_{\partial} \times {\rm id})^*
\ev^* \, \alpha
\nonumber \\
&&
+ \left(\mu_* \ev^* \Theta 
- \frac{1}{2} \mu_* \ev^* \sbv{\sbv{\vartheta}{\alpha}}{\alpha}
\cdots \right)|_{\delta_{\calL^{\perp}} S=0},
\label{akszonboundary2}
\end{eqnarray}
where $S_{\calL, \partial\calX}$ is a Q-structure function 
in the AKSZ construction, 
$S_{\calL, \calX}$ is an extra term from
the integration over $\calX$.
(Note that $\mu_*$ is the integration over $\calX$.)
The last terms are physically called Wess-Zumino terms. 
In general, the master equation is not satisfied because of the presence of $S_{\calL, \calX}$.
If $\sbv{\alpha}{\alpha}=0$
but $\sbv{\alpha_{\calL}}{\alpha_{\calL}}_{\Theta}\neq 0$,
the WZ terms can be written as the following expression
by using Equation (\ref{AKSZbulk}):
\begin{eqnarray}
S_{\calL} 
&=& \iota_{\hat{D}_{\partial X}}
\mu_{\partial \calX *} {\rm ev}^*
\vartheta_{\Theta}
- \mu_{\partial \calX *} (i_{\partial} \times {\rm id})^*
\ev^* \, \alpha
+ \mu_* \ev^* \Theta|_{\delta_{\calL^{\perp}} S=0}.
\label{akszonboundary3}
\end{eqnarray}

We call this construction of a topological sigma model,
$(\Map(\partial \calX, \calL), \sbv{-}{-}_{\calL}, S_{\calL})$, 
the twisted AKSZ construction, and we call the resulting model 
the twisted AKSZ sigma model.
The following schematic diagram expresses our construction procedure:
$$
\begin{CD}
\QPM_{T^*[n+1]\calL} @> >> \AKSZ_{\calX, T^*[n+1]\calL} \\
@V  VV @V  VV \\
\tQPM^0_{\calL} @>  >> \tAKSZ^0_{\partial \calX, \calL}
\end{CD}
$$
Here
$\tQPM^0_{\calL}$ is a twisted QP manifold on $\calL$
and 
$\tAKSZ^0_{\partial \calX, \calL}$ is a twisted AKSZ sigma model
defined on the boundary $\partial\calX$.
In one method, we construct an AKSZ sigma model with a boundary 
and then reduce the theory to the boundary.
In the other method, we reduce a QP pair 
to a small twisted QP manifold and then construct a twisted AKSZ sigma model.
The two methods produce the same twisted 
AKSZ sigma model $\tAKSZ^0_{\partial \calX, \calL}$.

There is an ambiguity in the definition of an AKSZ sigma model pair because 
there is an ambiguity in how we should take a big QP manifold for a fixed twisted QP manifold.
The cohomology class of $\Theta$ for the big bracket 
determines the equivalence class of this ambiguity.
In the language of sigma models, 
if $S_{{\sc bulk}}|_{\partial \calX}=0$
for $S^{\prime}_{T^*[n+1]\calL} = S_{T^*[n+1]\calL} + S_{{\sc bulk}}$
on $\calX$,
then $S$ and $S^{\prime}$ define the same 
boundary in the twisted AKSZ sigma models.

We will now show some examples of twisted AKSZ sigma models.

\begin{example}[$n=2$: Twisted-(or WZ-)Poisson Sigma Models]
\label{twistedPSM}
\noindent
We consider a QP pair for the case of 
$\Theta_C=\Theta_R=\rho=0$ in 
Example \ref{twistedPoissonexample}.
The big QP manifold is $T^*[2]\calL=T^*[2]T[1]M$,
and the QP structure defines the twisted Poisson structure
$[\pi_{\calL}, \pi_{\calL}]_S = \wedge^3 \pi^{\#} H$.

Let us take a three-dimensional manifold $X$ for which the
boundary is a two-dimensional manifold
$\partial X$.
The AKSZ sigma model 
is defined on $\Map(T[1]X, T^*[2]T[1]M)$ as follows.
The P-structure is
\begin{eqnarray*}
\bomega &=& \int_{\calX}
\mu
\
(
\delta \bbx^{i} \wedge \delta \bbxi_{i} +
\delta \bp_i \wedge \delta \bq^{i}
),
\end{eqnarray*}
where the boldface letters are superfields induced from
the pullbacks by $\bbx^*$ of corresponding local coordinates.
If $\alpha =0$,
the Q-structure function has the following form:
\begin{eqnarray}
S&=& S_0 + S_1 = \int_{\calX}
\mu
\ \left(- \bbxi_{i} \bbd \bbx^{i}
+ \bbq^{i} \bbd \bbp_{i}
+
\bbxi_{i} \bq^{i}
+ \frac{1}{3!}
H_{ijk}(\bbx) \bq^{i} \bq^{j} \bq^{k}
\right).
\end{eqnarray}

In order to obtain the equations of motion
from the variational principle,
we take the variation of $S$,
\begin{eqnarray}
\delta S =
\int_{\calX}
\mu
\ \left(- \delta \bbxi_{i} \bbd \bbx^{i}
- \bbxi_{i} \bbd \delta \bbx^{i}
+ \delta \bbq^{i} \bbd \bbp_{i}
+ \bbq^{i} \bbd \delta \bbp_{i}
+
\delta
\left(\bbxi_{i} \bq^{i}
+ \frac{1}{3!}
H_{ijk}(\bbx) \bq^{i} \bq^{j} \bq^{k}
\right)
\right).
\end{eqnarray}
The equations of motion of $\bbxi$ and $\bbq$ are obtained by 
integration by parts.
Therefore, the boundary terms,
$- \bbxi_{i} \bbd \delta \bbx^{i} + \bbq^{i} \bbd \delta \bbp_{i}$,
must vanish:
\begin{eqnarray}
\delta S|_{\partial \calX} =
\int_{\partial \calX}
\mu_{\partial \calX}
\ \left(
- \bbxi_{i} \delta \bbx^{i}
- \bbq^{i} \delta \bbp_{i}
\right)|_{\partial \calX} = 0.
\label{boundaryintegration}
\end{eqnarray}
This determines the boundary conditions.
Equation (\ref{boundaryintegration})
is satisfied if $\vartheta = 0$
on ${\rm Im} \ \partial \calX$,
i.e., ${\rm Im} \ \partial \calX \subset \calL$.
Typically, two kind of boundary conditions are possible:
$\bbxi_{//i} = 0$ or $\delta \bbx_{//}^{i}=0$, and
$\bbq_{//}^{i} =0$ or $\delta \bbp_{//i}=0$,
where $//$ is the component that is parallel to the boundary.
(Hybrid boundary conditions are also possible.)

Let us take boundary conditions
$\bbxi_{//i} = 0$ and $\bbq_{//}^{i} =0$
as an example.
Another condition is that the boundary conditions
must be consistent with the classical master equation
$\sbv{S}{S}=0$.
Direct computation gives
\begin{eqnarray}
\sbv{S}{S} =
\int_{\partial \calX}
\mu_{\partial \calX}
\ \left.
\left(- \bbxi_{i} \bbd \bbx^{i}
+ \bbq^{i} \bbd \bbp_{i}
+ \bbxi_{i} \bq^{i}
+ \frac{1}{3!}
H_{ijk}(\bbx) \bq^{i} \bq^{j} \bq^{k}
\right) \right|_{\partial \calX}.
\label{boundarymasterequation}
\end{eqnarray}
Due to the boundary conditions
$\bbxi_{//i} = 0$ and $\bbq_{//}^{i} =0$,
the first two terms corresponding to $\vartheta$ 
of the right-hand side
vanish on the boundary:
\begin{eqnarray}
\int_{\partial \calX}
\mu_{\partial \calX}
\
(i_{\partial} \times {\rm id})^* \ev^*
\vartheta =
\int_{\partial \calX}
\mu_{\partial \calX}
\ \left(- \bbxi_{i} \bbd \bbx^{i}
+ \bbq^{i} \bbd \bbp_{i}
\right)|_{\partial \calX} = 0.
\label{kinticboundary}
\end{eqnarray}
Therefore, the last two terms corresponding to the $\Theta$ terms
in Equation (\ref{boundarymasterequation})
must vanish:
\begin{eqnarray}
\int_{\partial \calX}
\mu_{\partial \calX}
 \
(i_{\partial} \times {\rm id})^* \ev^* \Theta =
\int_{\partial \calX}
\mu_{\partial \calX}
\ \left. \left(
\bbxi_{i} \bq^{i}
+ \frac{1}{3!}
H_{ijk}(\bbx) \bq^{i} \bq^{j} \bq^{k}\right) \right|_{\partial \calX}=0.
\label{interactionboundary}
\end{eqnarray}
The Q-structure 
follows immediately from the boundary conditions
$\bbxi_{//i} = 0$ and $\bbq_{//}^{i} =0$.

This consistency of the boundary conditions
is described in terms of a target QP manifold $T^*[2]\calL$. 
Equation (\ref{kinticboundary}) is satisfied if
$\vartheta|_{\calL}=0$.
Under this condition, 
Equation (\ref{interactionboundary}) is
satisfied if $\ev^* \Theta|_{\partial \calX}=0$.
This condition is the pullback of the equation
$\Theta|_{\calL} = 0$.
This corresponds to Proposition \ref{boundarytheta2}.

Next we introduce $\alpha$.
The Q-structure is modified by introducing a canonical function $\alpha$.
For example, we take
$\int_{\partial \calX}
\mu_{\partial \calX}
\
(i_{\partial} \times {\rm id})^* \ev^* \alpha =
\int_{\partial \calX}
\mu_{\partial \calX}
\
\frac{1}{2}
\pi^{ij}(\bbx)
\bp_{i} \bp_{j}$. 
The Q-structure changes to
\begin{eqnarray}
S^{\prime}
&=& \int_{\calX}
\mu
\ \left(- \bbxi_{i} \bbd \bbx^{i}
+ \bbq^{i} \bbd \bbp_{i}
+
\bbxi_{i} \bq^{i}
+ \frac{1}{3!}
H_{ijk}(\bbx) \bq^{i} \bq^{j} \bq^{k}
\right)
\nonumber \\
&& -
\int_{\partial \calX}
\mu_{\partial \calX}
\
\frac{1}{2}
\pi^{ij}(\bbx)
\bp_{i} \bp_{j}.
\label{3DCSMboundary}
\end{eqnarray}
The boundary term deforms the boundary conditions.
The variation $\delta S$
is changed to
\begin{eqnarray*}
\delta S^{\prime}|_{\partial \calX}
=
\int_{\partial \calX}
\mu_{\partial \calX}
\ \left[\left(- \bbxi_{i}
+ \frac{1}{2} \frac{\partial \pi^{ij}(\bbx)}{\partial \bbx^i}
\bp_{j} \bp_{k} \right)
\delta \bbx^{i}
+ \left(- \bbq^{i}
-
\pi^{ij}(\bbx) \bp_{j}
\right)
\delta \bbp_{i}
+ \cdots
\right].
\end{eqnarray*}
Since these terms must vanish, consistent boundary conditions are as follows:
\begin{eqnarray}
\bbxi_{i}|_{//} =
- \frac{1}{2} \frac{\partial \pi^{jk}}{\partial \bbx^i}(\bbx)
\bp_{j} \bp_{k}|_{//},
\qquad
\bq^i|_{//} &=& \pi^{ij}(\bbx) \bp_{j} |_{//}.
\label{WZPoissonboundary}
\end{eqnarray}
$\sbv{S^{\prime}}{S^{\prime}}=0$ requires
another consistency condition, 
i.e., the integrand of $S_1$ is zero on the boundary:
\begin{eqnarray}
\left.
\left(\bbxi_{i} \bq^{i}
+ \frac{1}{3!}
H_{ijk}(\bbx) \bq^{i} \bq^{j} \bq^{k}
\right) \right|_{//}
= 0.
\label{boundaryS1}
\end{eqnarray}
Similarly, (\ref{WZPoissonboundary}) and (\ref{boundaryS1}) can be
expressed by the condition on $T^*[2]\calL=T^*[2]T[1]M$.
The condition is that
\begin{eqnarray}
\xi_{i} q^{i}
+ \frac{1}{3!}
H_{ijk}(x) q^{i} q^{j} q^{k}
= 0
\label{WZPS1}
\end{eqnarray}
on the Lagrangian submanifold $\calL^{\prime}
$ defined by
\begin{eqnarray}
\xi_{i} = - \frac{1}{2} \frac{\partial \pi^{jk}}{\partial x^i}(x)
p_{j} p_{k},
\qquad
q^i = \pi^{ij}(x) p_{j}.
\label{WZPLagrangian}
\end{eqnarray}
Here we consider the symplectic form graded by $e^{-\delta_\alpha}$
while preserving $Q$. $\calL^{\prime}$ is a Lagrangian submanifold 
with respect to $e^{-\delta_\alpha} \omega$.
Substituting (\ref{WZPLagrangian}) into
(\ref{WZPS1}), we obtain the geometric structure on
$\calL^{\prime}$:
\begin{eqnarray}
&& \xi_{i} q^{i}
+ \frac{1}{3!}
H_{ijk}(x) q^{i} q^{j} q^{k}
\nonumber \\
&& =
- \frac{1}{2} \frac{\partial \pi^{jk}}{\partial x^l}(x)
\pi^{li}(x) p_{j} p_{k} p_{i}
+ \frac{1}{3!}
H_{ijk}(x) \pi^{il}(x)
\pi^{jm}(x) \pi^{kn}(x) p_{l} p_{m} p_{n} =0.
\label{twistedPoisson}
\end{eqnarray}
This condition corresponds to Proposition \ref{boundaryQstr}.
Equation (\ref{twistedPoisson}) is nothing but the equation of
the twisted Poisson structure $[\pi,\pi]_S = \wedge^3 \pi^{\#} H$.

Physically, the boundary action on $\partial T[1]X$
is obtained by integrating out the superfield $\bbxi_i$
from Equation (\ref{3DCSMboundary}).
By integrating out $\bbxi_i$, we obtain the equations of motion 
$\bbq^i = \bbd \bbx^i$.
By restricting this orbit,
we obtain the twisted-Poisson sigma model
on $\calL^{\prime}$:
\begin{eqnarray}
S_{\calL^{\prime}}
&=&
\int_{\partial \calX}
\mu_{\partial \calX} \
\left(
\bp_{i} \bbd \bbx^{i}
-
\frac{1}{2}
\pi^{ij}(\bbx)
\bp_{i} \bp_{j}
\right)
+
\int_{\calX}
\mu_{\calX} \
\frac{1}{3!}
H_{ijk}(\bbx)
\bbd \bbx^{i}
\bbd \bbx^{j}
\bbd \bbx^{k}
.
\end{eqnarray}
This coincides with the physical model constructed from 
the twisted Poisson structure on $\partial X$
\cite{Klimcik:2001vg}, and
it is the twisted AKSZ sigma model constructed from 
the twisted Poisson structure on $\calL^{\prime}$.
The P-structure on $\Map(T[1]X, \calL^{\prime})$ 
is equivalent to the one constructed from the derived bracket
$\sbv{-}{-}_{\Theta}=\sbv{\sbv{-}{\Theta}}{-}|_{\Map(\partial \calX, \calL^{\prime})}$.
%
The $H$ term breaks the classical master equation,
i.e.,
$\sbv{S_{\calL^{\prime}}}{S_{\calL^{\prime}}}_{\Theta} \neq 0$.

\end{example}

\begin{example}[$n =3$: The Twisted Strong Courant Sigma Models]
\noindent
Let us consider the AKSZ sigma model induced from a
QP manifold of degree $3$ and its canonical function.
We take a QP manifold of degree $3$,
$T^*[3] \calL = T^*[3] (T^*[2] E[1] \oplus A[2])$,
as a generalization of Section \ref{sectionSCA}, 
where $E$ and $A$ are two vector bundles on $M$.

We take the local coordinates on $T^*[2] E[1] \oplus A[2]$
as $(x^i, u^a, z_p, p_i)$
of degree $(0,1,2,2)$, where
$x^i$ is a local coordinate on $M$,
$u^a$ is on the fiber of $E$,
$z_p$ is on the fiber of $A$, and
$p_i$ is on the fiber of $T^*[2]M$.
The conjugate local coordinates of the fiber are
$(\xi_i, v_a, w^p, q^i)$ of degree $(3, 2, 1, 1)$.
A graded symplectic structure is given by
$\omega = \delta x^i \wedge \delta \xi_i
- \delta u^a \wedge \delta v_a
+ \delta p_i \wedge \delta q^i
+ \delta z_p \wedge \delta w^p
$.
We consider the following Q-structure
satisfying $\Theta |_{\calL}=0$:
\bea
\Theta &=&
\xi_i q^i
+ \frac{1}{2}k^{ab} v_a v_b
+ \rho^i{}_r(x) \xi_{i} w^{r}
+ \frac{1}{2}
C^r{}_{pq}(x) z_r w^p w^q
+\frac{1}{4!}
H_{ijkl}(x) q^{i} q^{j} q^{k} q^{l},
\eea
where
$k^{ab}$ is a fiber metric on $E$ and
$H = \frac{1}{4!} H_{ijkl}(x) dx^{i} \wedge dx^{j} \wedge dx^{k}
\wedge dx^{l}$
is a $4$-form on $M$.
If $A$ is a Lie algebroid and $H$ is closed,
$\sbv{\Theta}{\Theta}=0$ is satisfied.

Take the Lagrangian submanifold
$\calL
= T^*[2] E[1] \oplus A[2]
= \{ \xi_i = v_a  = q^i = w^p = 0 \}
$
and a function of degree $3$,
\bea
\alpha =
\sigma^{i}{}_a(x) p_{i} u^{a}
+ \tau^r{}_{a}(x) z_r u^a
+\frac{1}{3!}h_{abc}(x) u^a u^b u^c.
\eea
The canonical function equation
$e^{\delta_{\alpha}}\Theta|_{\calL}=0$ determines the identities
among $\sigma^{i}{}_a(x)$,
$\tau^r{}_{a}(x)$, and $h_{abc}(x)$
and the geometric conditions on $\calL$.

If $w^p=z_p=0$ and $H=0$, this reduces to the Courant algebroid.
If $w^p=z_p=0$ and $H \neq 0$, this reduces to the $H_4$-twisted
Courant algebroid
in Example \ref{twistedCourantexample}.
If $q^i=p_i=0$, this reduces to the strong Courant algebroid.

Let us take a four-dimensional manifold $X$ with a boundary 
that is a three-dimensional manifold
$\partial X$.
The bulk AKSZ sigma model on
$T^*[3] (T^*[2] E[1] \oplus A[1])$
is constructed in the usual way.
The P-structure is
\begin{eqnarray}
\bomega &=& \int_{\calX}
\mu
\
(
\delta \bbx^{i} \wedge \delta \bbxi_{i}
+ \delta \bp_i \wedge \delta \bq^{i}
- \delta \bu^a \wedge \delta \bv_a
+ \delta \bz_p \wedge \delta \bw^p
).
\end{eqnarray}
The Q-structure function with boundary terms has the following form:
\begin{eqnarray}
S&=& \int_{\calX}
\mu
\ \left(
\bbxi_{i} \bbd \bbx^{i}
+ \bbq^{i} \bbd \bbp_{i}
- \bv_a \bbd \bu^a
+ \bw^p \bbd \bz_p
+ \bbxi_{i} \bq^{i}
\right.
\nonumber \\
&&
\left.
+ \frac{1}{2} k^{ab} \bv_a \bv_b
+ \rho^i{}_r(\bbx) \bbxi_{i} \bw^{r}
+ \frac{1}{2} C^r{}_{pq}(\bbx) \bz_r \bw^p \bw^q
+ \frac{1}{4!}
H_{ijkl}(\bbx) \bq^{i} \bq^{j} \bq^{k} \bq^l
\right)
\nonumber \\
&&
- \int_{\partial \calX}
\mu_{\partial X}
\left(
\sigma^{i}{}_a(\bbx) \bp_{i} \bu^{a}
+ \tau^r{}_{a}(\bbx) \bz_r \bu^a
+ \frac{1}{3!} h_{abc}(\bbx) \bu^a \bu^b \bu^c
\right).
\label{4DTSCourantSigmaModel}
\end{eqnarray}
We can construct the boundary of the twisted AKSZ sigma model
by using the method in this section.
In this example, the derived Poisson bracket on $\calL$
is generally degenerate.
In degenerate cases, it is difficult to express
concretely the boundary Q-structure function $S$.
However, the bulk and boundary theories are physically consistent
because the bulk classical master equation is satisfied.

If $A=0$, the QP manifold is $T^*[3] \calL = T^*[3] T^*[2] E[1]$, and
the derived Poisson bracket is nondegenerate.
The corresponding twisted QP manifold $T^*[2] E[1]$
produces the 
boundary topological sigma model with a Wess-Zumino term
from a three-dimensional manifold $\partial X$
to the target space $T^*[2] E[1]$:
\begin{eqnarray}
S_{\calL}&=& \int_{\partial \calX}
\mu_{\partial \calX}
\ \left(\bp_{i} \bbd \bbx^{i}
- \frac{1}{2} k_{ab} \bu^{a} \bbd \bu^{b}
- \sigma^i{}_a(\bbx)
\bp_{i} \bu^{a}
- \frac{1}{3!}
h_{abc}(\bx) \bu^{a} \bu^{b} \bu^{c}
\right)
\nonumber \\
&& +
\int_{\calX}
\mu
\
\frac{1}{4!}
H_{ijkl}(\bbx)
\bbd \bbx^{i}
\bbd \bbx^{j}
\bbd \bbx^{k}
\bbd \bbx^{l}.
\end{eqnarray}
This twisted Q-structure function can also be obtained by
integrating out $\bbxi_i$ and $\bv_a$
in Equation (\ref{4DTSCourantSigmaModel}) by physical arguments.
This is equivalent to 
the $H_4$-twisted Courant sigma model \cite{Hansen:2009zd}.

\end{example}

\section{Summary and Future Areas of Work}
\noindent 
In this paper, we have analyzed the mathematical structures of the 
boundary conditions of AKSZ sigma models and boundary theories,
and we have discussed the connections between canonical functions and QP pairs. 
We have proposed a new algebroid, the strong Courant algebroid.

However, we have not completely analyzed the properties of the solutions
of the canonical function equations, the deformation theory,
and the equivalence classes of the QP description of a twisted QP manifold.
The deformation theory of QP structures will unify various
classical and new structures,
for examples, the Nijenhuis structures on Lie algebras, the
Courant-like algebroid structures in Poisson geometry, 
and others
\cite{Kosmann-Schwarzbach:1990}.
In order to understand the complete structures,
we should consider the general classification of
a QP pair $(T^*[n+1]\calL, \omega, \Theta, \calL, \alpha)$
and $(\calL,\omega_{\calL},\alpha_{\calL})$, 
which can be viewed as a simultaneous deformation 
of the homological function $\Theta$ and the canonical function $\alpha$.

In this paper, all of the examples derived from
canonical functions are geometric.
If we consider a QP manifold and a canonical function over a point,
we can derive other kinds of algebraic structures.
We will leave these as areas to investigate in future work.

The quantization of AKSZ sigma models with boundaries is 
the next step.
Since AKSZ sigma models are topological field theories,
a bulk AKSZ sigma model
and the corresponding boundary theory should 
have the same physical partition functions
and produce equivalent mathematical and physical information.
This will not only describe the quantum membrane theories in physics
but will also quantize a wide class of geometric structures.

\subsection*{Acknowledgments}
\noindent
The authors would like to thank J. Stasheff for valuable discussions
and comments.
N.~Ikeda would like to thank the organizers of The 33rd Winter School
Geometry and Physics, Srni, Czech Republic,
where part of this work was performed.
X.-M.~Xu is grateful to Zhangju Liu for his patient
guidance throughout his study at Peking
University. He also would like to thank Jianghua Lu for hospitality during his visit to the University of Hong Kong, where
part of this work was performed.
X.-M.~Xu is supported by grant number PDFMP2-141756 of 
the Swiss National Science Foundation.


\newcommand{\bibit}{\sl}



\end{document}